\documentclass[parskip=half]{scrartcl}
\usepackage[utf8]{inputenc}
\usepackage[T1]{fontenc}
\usepackage{lmodern}
\usepackage{amsmath,amssymb,amsthm,amsfonts,bm,mathtools}
\usepackage{enumerate,color}
\usepackage[usenames,dvipsnames]{xcolor}
\usepackage{hyperref}
\hypersetup{
  colorlinks=true,
  linkcolor=black,
  citecolor=black,
  urlcolor=blue,
  pdftitle={Global solvability of a model for tuberculosis granuloma formation},
  pdfauthor={Fuest, Lankeit, Mizukami},
  pdfkeywords={},
  bookmarksopen=true,
}

\usepackage[numbers]{natbib}

\makeatletter
\renewenvironment{proof}[1][\proofname]{\par
  \pushQED{\qed}%
  \normalfont \topsep0\p@\relax
  \trivlist
  \item[\hskip\labelsep\scshape
  #1\@addpunct{.}]\ignorespaces
}{%
  \popQED\endtrivlist\@endpefalse
}
\makeatother

\numberwithin{equation}{section} 
\newtheorem{thm}{Theorem}[section]

\newtheorem{lem}[thm]{Lemma}

\theoremstyle{definition}
\newtheorem{df}[thm]{Definition}
\newtheorem{remark}[thm]{Remark}
\newtheorem*{prth1.1}{Proof of Theorem 1.1}

\usepackage{newunicodechar}
\newunicodechar{ℝ}{\mathbb{R}}
\newunicodechar{λ}{\lambda}
\newunicodechar{σ}{\sigma}
\newunicodechar{ε}{\varepsilon}
\newunicodechar{∇}{\nabla}
\newunicodechar{φ}{\varphi}

\newcommand{\Du}{D_u}
\newcommand{\Dv}{D_v}
\newcommand{\Dw}{D_w}
\newcommand{\Dz}{D_z}
\newcommand{\chiu}{\chi_u}
\newcommand{\chiz}{\chi_z}
\newcommand{\betau}{\beta_u}
\newcommand{\gammau}{\gamma_u}
\newcommand{\gammav}{\gamma_v}
\newcommand{\gammaw}{\gamma_w}
\newcommand{\deltau}{\delta_u}
\newcommand{\deltaz}{\delta_z}
\newcommand{\muv}{\mu_v}
\newcommand{\muw}{\mu_w}
\newcommand{\alphaw}{\alpha_w}
\newcommand{\alphaz}{\alpha_z}
\newcommand{\betaz}{\beta_z}
\newcommand{\rhov}{\rho_v}
\newcommand{\Sigmaset}{\sum} 

\newcommand{\tmax}{T_{\mathrm{max}}}
\newcommand{\tmaxe}{T_{\max, \eps}}
\newcommand{\lp}[2]{\|#2\|_{L^{#1}(\Omega)}}

\newcommand{\Ombar}{\overline{\Omega}}
\newcommand{\Ombarinf}{\Ombar \times [0, \infty)}

\newcommand{\loc}{\mathrm{loc}}
\newcommand{\embed}{\hookrightarrow}

\newcommand{\intnt}{\int_0^T}

\newcommand{\intntom}{\int_0^T \!\! \int_\Omega}

\newcommand{\intninfom}{\int_0^\infty \!\! \int_\Omega}

\newcommand{\leb}[2][\Omega]{\ensuremath{L^{#2}(#1)}}
\newcommand{\lebl}[1][\Omega]{\ensuremath{L\log L(#1)}}
\newcommand{\sob}[3][\Omega]{\ensuremath{W^{#2, #3}(#1)}}
\newcommand{\con}[2][\Ombar]{\ensuremath{C^{#2}(#1)}}

\newcommand{\io}{\int_\Omega}

\newcommand{\iio}{\int_0^T\!\! \io}

\newcommand{\ol}{\overline}

\newcommand{\ds}{\,\mathrm{d}s}

\newcommand{\dtau}{\,\mathrm{d}\tau}
\newcommand{\dsigma}{\,\mathrm{d}\sigma}

\newcommand{\hp}{\hphantom}
\newcommand{\pe}{\mathrel{\hp{=}}}

\newcommand{\R}{\mathbb{R}}

\newcommand{\N}{\mathbb{N}}

\newcommand{\ur}[1]{\mathrm{#1}}
\newcommand{\ure}{\ur e}
\newcommand{\ddt}{\frac{\mathrm{d}}{\mathrm{d}t}}

\newcommand{\eps}{\varepsilon}
\newcommand{\intom}{\int_\Omega}

\newcommand{\defs}{\coloneqq}

\newcommand{\nea}{\nearrow}

\newcommand{\sea}{\searrow}

\newcommand{\rh}{\rightharpoonup}

\newcommand{\ue}{u_\eps}
\newcommand{\uet}{u_{\eps t}}
\newcommand{\une}{u_{0 \eps}}

\newcommand{\ve}{v_\eps}
\newcommand{\vet}{v_{\eps t}}
\newcommand{\vne}{v_{0 \eps}}

\newcommand{\we}{w_\eps}
\newcommand{\wet}{w_{\eps t}}
\newcommand{\wne}{w_{0 \eps}}

\newcommand{\ze}{z_\eps}
\newcommand{\zet}{z_{\eps t}}
\newcommand{\zne}{z_{0 \eps}}

\newcommand{\sigmae}{\sigma_\eps}

\newcommand{\f}[2]{\frac{#1}{#2}}

\usepackage{authblk}

\author[1]{Mario~Fuest\footnote{e-mail: fuest@ifam.uni-hannover.de, ORCID: 0000-0002-8471-4451}}
\author[1]{Johannes~Lankeit\footnote{e-mail: lankeit@ifam.uni-hannover.de, ORCID: 0000-0002-2563-7759}}
\author[2]{Masaaki~Mizukami\footnote{e-mail: masaaki.mizukami.math@gmail.com, ORCID: 0000-0002-5496-5129}}

\affil[1]{Leibniz Universität Hannover, Institut für Angewandte Mathematik, \protect\\ Welfengarten 1, 30167 Hannover, Germany}
\affil[2]{Department of Mathematics, Faculty of Education, 
Kyoto University of Education, \protect\\ 1, Fujinomori, Fukakusa, Fushimi-ku, Kyoto 612-8522, Japan} 

\date{}

\title{Global solvability of a model for tuberculosis granuloma formation}

\begin{document}
\setkomafont{title}{\normalfont\Large}
\maketitle

\KOMAoptions{abstract=true}
\begin{abstract}
\noindent
We discuss a nonlinear system of partial differential equations modelling the formation of granuloma during tuberculosis infections and prove the global solvability of the homogeneous Neumann problem for 
\begin{align*}
  \begin{cases}
    u_t = D_u \Delta u - \chi_u \nabla \cdot (u \nabla v) - \gamma_u uv - \delta_u u + \beta_u, \\
    v_t = D_v \Delta v + \rho_v v - \gamma_v uv + \mu_v w,\\
    w_t = D_w \Delta w + \gamma_w uv - \alpha_w wz - \mu_w w,\\
    z_t = D_z \Delta z - \chi_z \nabla \cdot (z \nabla w) + \alpha_z f(w)z - \delta_z z
  \end{cases}
\end{align*}
in bounded domains in the classical and weak sense in the two- and three-dimensional setting, respectively.\\[2pt]
In order to derive suitable a~priori estimates, we study the evolution of the well-known energy functional for the chemotaxis--consumption system both for the $(u, v)$- and the $(z, w)$-subsystem.
A key challenge compared to “pure” consumption systems consists of overcoming the difficulties raised by the additional, in part positive, terms in the second and third equations.
This is inter alia achieved by utilising a dissipative term of the (quasi-)energy functional, which may just be discarded in simpler consumption systems.
\\[5pt]
 \textbf{Mathematics Subject Classification (MSC 2020):} 35K55 (primary), 35A01, 35A09, 35D30, 35Q92, 92C17, 92C50 (secondary)
\end{abstract}

\section{Introduction}
Tuberculosis (TB) is an infectious disease that has plagued humankind since ancient times \cite{ancient_tuberculosis} and even nowadays is thought to be present as latent infection in almost a quarter of the world's population \cite{houben_dodd}. 
Even though it is preventable and usually curable, in 2022 it was the second leading cause of death from a single infectious agent (after COVID-19), \cite{who_tuberculosis_rept2023}. 
Mathematical modeling of tuberculosis has largely concentrated on an epidemiological, population level perspective%
, see \cite{waaler,castillo-chavez_song,feng2000model,feng_ianelli_milner,ozacglar,zhang_liu_feng_jin}. 
The reaction of the immune system within a single host was considered in, e.g., \cite{antia1996models,magombedze,ibarguen_esteva_chavez,ibarguen_esteva_burbano} or in the context of co-infection of HIV and tuberculosis in an ODE model for T cells, macrophages and pathogens in \cite{kirschner1999dynamics}; further related co-infection studies and an overview of their results can be found in \cite[table 1, p.12]{hoerter}.

In the present article, we focus on an even more detailed level of description (based on the model from \cite{feng}), in order to capture spatially inhomogeneous phenomena, like the formation of so-called granuloma, one of the distinguishing features of tuberculosis infections. (For an overview of different types of within-host models also see the survey \cite{chakraborty2024brief}.)

When previously aerosol-borne \textit{Mycobacterium tuberculosis} (\textit{M.tb.}) enters the body of a host, primary TB develops in the alveoli of the lung (see e.g.\ \cite{wigginton_kirschner,ehlers_schaible,gammack_ganguli_marino_kirschner} for a more detailed description
). Alveolar macrophages approach 
and phagocytose \textit{M.tb.}, but fail to kill them. Instead, they turn into infected macrophages, in whose interior the bacteria multiply. At the same time, they release chemokines that attract cells forming the adaptive immune response (T  lymphocytes). These necrotize infected macrophages. An observable consequence of this process is the development of granuloma, roughly spherical aggregations consisting of bacteria, macrophages and other immune cells. Seemingly protecting the host by walling of the infection, on the other hand granuloma provide a niche promoting bacterial proliferation \cite{ramakrishnan2012revisiting}. 

Following \cite{feng}, the PDE system to be studied describes the spatio-temporal evolution of the concentration of (healthy) macrophages ($u$) and (extracellular) bacteria ($v$), infected macrophages ($w$) and lymphocytes ($z$), as most important contributors to the process. In order to keep the amount of components low and the system analytically tractable, additional, less important factors (separate chemical signal substances, further differentiated types of immune cells) are neglected. The model then reads 

\begin{align}\label{P}
  \begin{cases}
    u_t = \Du \Delta u - \chiu \nabla \cdot (u \nabla v) - \gammau uv - \deltau u + \betau & \text{in $\Omega \times (0, \infty)$}, \\
    v_t = \Dv \Delta v + \rhov v - \gammav uv + \muv w                                     & \text{in $\Omega \times (0, \infty)$},\\
    w_t = \Dw \Delta w + \gammaw uv - \alphaw wz - \muw w                                  & \text{in $\Omega \times (0, \infty)$},\\
    z_t = \Dz \Delta z - \chiz \nabla \cdot (z \nabla w) + \alphaz f(w)z - \deltaz z       & \text{in $\Omega \times (0, \infty)$}, \\
    \partial_\nu u = \partial_\nu v = \partial_\nu w = \partial_\nu z = 0                  & \text{on $\partial \Omega \times (0, \infty)$}, \\
    (u, v, w, z)(\cdot, 0) = (u_0, v_0, w_0, z_0)                                          & \text{in $\Omega$},
  \end{cases}
\end{align}
where
\begin{align}\label{eq:intro:params}
  \Du, \Dv, \Dw, \Dz,
  \chiu, \chiz,
  \gammau, \gammav,\gammaw,
  \muv, \muw,
  \alphaw, \alphaz,
  \deltau, \rhov, \deltaz,
  \betau
  > 0
\end{align}
and $\Omega$ is a sufficiently smooth bounded domain in $ℝ^2$ or $ℝ^3$  and $f(w)=w$ (for generalizations, see \eqref{cond:f} below). 

Effects 
modeled by \eqref{P} are the following: Healthy macrophages (of density $u$), produced with constant rate $\betau$ and finite life expectancy (determined by $\deltau$), move randomly (with effective diffusivity $\Du$) and are chemotactically attracted by bacteria \textit{M.~tb.~}(with chemotactic sensitivity $\chiu$), upon contact with which they are infected and turned into infected macrophages (the strength of this effect is given by $\gammau$). Bacteria ($v$) diffuse ($\Dv$) and replicate (with rate $\rhov$), are engulfed when encountering healthy macrophages ($\gammav$) and released from infected macrophages, the amount of bacteria released by macrophages encoded in $\muv$. Infected macrophages ($w$) diffuse ($\Dw$), are created from engulfment of bacteria by healthy macrophages ($\gammaw$), are necrotized by lymphocytes upon contact ($\alphaw$) and die with rate $\muw$. Lymphocytes diffuse ($\Dz$), are chemotactically attracted by infected macrophages ($\chiz$) and die with rate $\deltaz$. Their reproduction rate $\alphaz f(w)$ depends on the presence of infected macrophages. In \cite{feng}, a log-sigmoidal form of T cell activation was chosen, so that the replication term reads $\alphaz z \f{w}{\betaz +w}$ with some $\betaz>0$. The important effect is that the more infected macrophages there are, the more T cells are needed and produced. We encode this effect in \eqref{P} in the term $+\alphaz f(w)z$, and admit the choice $f(w)=w$, which is more idealized and, and from the point of view of existence analysis, seems to be the worst-growing biologically sensible case. 

In \cite{feng}, \eqref{P} (with $f(w)=\f{w}{\betaz+w}$) was investigated regarding stability of equilibria in the corresponding ODE system and \eqref{P} was numerically studied in one- and two-dimensional spatial domains, where the effect of stronger immune response (i.e.\ larger $\chiu$) was clearly visible (cf.\ \cite[Fig.\ 8--10]{feng}). This model stands in the tradition of \cite{gammack_doering_kirschner}, which presented a spatio-temporal model of the innate immune response (i.e.\ the initial response without participation of T cells), without stable latent solution. (For a discussion of the relation with models on different scales, see also \cite{gammack_ganguli_marino_kirschner}.) 
A different approach to granulomas (with analysis of their growth and stable states) is given by the free-boundary model of macrophages and bacteria of \cite{friedman_lam}. 
On the other hand, more involved models for immune response, with different types of macrophages, leukocytes, signal chemicals have been introduced, \cite{hao_schlesinger_friedman,catalaPLOS,su_zhou_dorman_jones}. In these cases, however, theoretical results beyond numerical experiments and possibly sensitivity analysis seem to be lacking.

Our aim in the present article is to show global existence of solutions of \eqref{P}:
\begin{thm}\label{TH;2D}
  Let $\Omega \subset \R^2$ be a smooth, bounded domain, suppose \eqref{eq:intro:params} and 
  \begin{equation}\label{cond:f}
   f\in C^1(ℝ), \qquad 0\le f(w)\le w \quad \text{for all }  w\in(0,\infty), 
  \end{equation}
  let $q > 2$
  and let
  \begin{align}\label{eq:2d_main:init}
    (u_0, v_0, w_0, z_0) \in \con0 \times \sob1q \times \sob1q \times \con0
    \quad \text{be nonnegative.}
  \end{align}
  Then there exists a global classical solution $(u, v, w, z)$ of \eqref{P}.
\end{thm}

In three-dimensional domains, the corresponding result is the following:

\begin{thm}\label{TH;3D}
  Let $\Omega \subset \R^3$ be a smooth, bounded domain, suppose \eqref{eq:intro:params} and \eqref{cond:f}
  and let
  \begin{align}\label{eq:3d_main:init}
    \begin{cases}
    (u_0, v_0, w_0, z_0) \in \lebl \times \leb\infty \times \leb3 \times \lebl \\
    \text{be nonnegative a.e.\ with }
    \frac{\nabla v_0}{\sqrt{v_0}}, \frac{\nabla w_0}{\sqrt{w_0}} \in \leb2.
    \end{cases}
  \end{align}
  Then there exists a global weak solution $(u, v, w, z)$ (in the sense of Definition~\ref{def:weak_sol} below) of \eqref{P}.
\end{thm}

\paragraph{Challenges compared to the chemotaxis--consumption system.}
Both the $(u, v)$- and the $(z, w)$-subsystem in \eqref{P} share some characteristics with the classic chemotaxis--consumption system
\begin{align}\label{prob:ct_con}
  \begin{cases}
    u_t = \Delta u - \nabla \cdot(u \nabla v), \\
    v_t = \Delta v - uv,
  \end{cases}
\end{align}
already proposed by Keller and Segel to model bacteria partially moving towards oxygen (\cite{KellerSegelTravelingBandsChemotactic1971});
see also \cite{LankeitWinklerDepletingSignalAnalysis2023} for a recent survey on \eqref{prob:ct_con} and close relatives thereof.
Although we ultimately obtain the same global well-posedness results as known for \eqref{prob:ct_con} (\cite{WinklerGlobalLargedataSolutions2012}),
that is, of classical solutions in two- and of weak solutions in three-dimensional settings,
the nonlinear coupling between these subsystems impedes a direct adaptation of the techniques from \cite{WinklerGlobalLargedataSolutions2012}.

While for \eqref{prob:ct_con} the maximum principle directly asserts boundedness of the signal,
this is no longer the case for \eqref{P}, as the latter contains the production terms $+\muv w$ and $+\gammaw uv$ in the second and third equation in \eqref{P}, respectively.
Mainly because a~priori estimates for $w$ in $L^{\frac{n}{2}+\eta}$ imply boundedness of $v$ and hence of $\intom uv$,
we can still show that $v$ and $w$ are bounded in $L^\infty$ and $L^p$ for $p<\frac{n}{n-2}$, respectively, see Lemma~\ref{lem;Linf-v}.
However, obtaining boundedness of $w$ seems to be out of reach in the three-dimensional setting
(essentially because it is entirely unclear how to obtain bounds for $u$ substantially stronger than those provided by the quasi-energy functional \eqref{eq:intro:functional} below).

Further a~priori estimates for \eqref{prob:ct_con} are then gained by the functional introduced in \cite{DuanEtAlGlobalSolutionsCoupled2010}, namely
\begin{align}\label{eq:intro:functional}
  \intom u \ln u + \io \frac{|\nabla v|^2}{v},
\end{align}
which decreases along trajectories if $\Omega = \R^n$ (\cite{DuanEtAlGlobalSolutionsCoupled2010}) or if $\Omega$ is smoothly bounded and convex (\cite{WinklerGlobalLargedataSolutions2012})
and is still a quasi-energy functional for general smooth, bounded domains (\cite{JiangEtAlGlobalExistenceAsymptotic2015}).
We follow this approach in Lemma~\ref{lem;firstEn} and Lemma~\ref{lem;secondEn} for the $(u, v)$- and the $(z, w)$-subsystem in \eqref{P}, respectively.
For the former, the additional term $+\muv w$ requires us to deal with the term $\intom \frac{|\nabla w|^2}{w}$ when differentiating \eqref{eq:intro:functional},
which, however, can be handled by adding $\intom w \ln w$ to the functional.
For the latter, we can then utilize a~priori estimates already obtained for $u$ and $v$.

Therefore, $\intom u \ln u$ and $\intom z \ln z$ remain bounded throughout (finite) evolution.
In two-dimensional settings, this implies global existence of classical solutions (cf.~\cite[Lemma 3.3]{BellomoEtAlMathematicalTheoryKeller2015}) and hence Theorem~\ref{TH;2D},
as we shall see in Section~\ref{sec:2d}.

For three-dimensional domains treated in Section~\ref{sec:3d}, we show that solutions to certain approximate problems converge to weak solutions of \eqref{P}.
To that end, the dissipative terms of the quasi-energy functional turn out to be crucial.
Regarding the $(z, w)$-subsystem, by following \cite{WinklerGlobalLargedataSolutions2012}, we have a~priori estimates such as
\begin{align}\label{eq:intro:space_time_est_1}
  \iio \frac{|\nabla z|^2}{z} \le C
  \quad \text{and} \quad
  \iio \frac{|\nabla w|^4}{w^3} \le C
\end{align}
at our disposal (cf.\ Lemma~\ref{lem;secondEn}).
Combining these with uniform-in-time estimates inter alia yields
\begin{align}\label{eq:intro:space_time_est_2}
  \iio z^\frac53 \le C, \quad
  \iio (wz)^{\frac54-\eta} \le C_\eta
  \quad \text{and} \quad
  \iio |\nabla w|^{\frac52-\eta} \le C_\eta, \quad
\end{align}
for all $\eta \in (0, \frac54)$, see Lemma~\ref{lm:spacetime_est}.
That is, even without an $L^\infty$ bound for the signal (which \cite{WinklerGlobalLargedataSolutions2012} can make use of),
we obtain space-time equi-integrability of $wz$.
Hence, corresponding products of the approximate solutions converge strongly in $L^1$, which takes care of the necrotization term $\intntom wz \varphi$ in Definition~\ref{def:weak_sol}.

The situation is worse for the taxis term $\intntom z \nabla w \cdot \nabla \varphi$.
If $w$ were bounded, $z \nabla w = z \cdot w^\frac34 \cdot w^{-\frac34} \nabla w$ would be equi-integrable by \eqref{eq:intro:space_time_est_1} and \eqref{eq:intro:space_time_est_2},
but, as discussed above, no such bound seems available.
Moreover, if \eqref{eq:intro:space_time_est_2} held for $\eta = 0$, then Young's inequality would also imply integrability of $z \nabla w$
(which could even be improved to equi-integrability by making use of boundedness of $\intom z \ln z$ instead of just $\intom z$, see for instance \cite[Corollary~1.2]{WinklerLogarithmicallyRefinedGagliardoNirenberg2023}).
However, it seems nontrivial to actually improve \eqref{eq:intro:space_time_est_2} to the endpoint case $\eta = 0$. 

Fortunately, the quasi-energy functional generates another dissipative term
(which has been simply discarded in \cite[Proof of Lemma~3.4]{WinklerGlobalLargedataSolutions2012}),
namely $\intntom \frac{|\nabla w|^2}{w} z$ (see Lemma~\ref{lem;secondEn}).
That is, $z \nabla w = \sqrt{wz} \cdot \frac{\sqrt{z}}{\sqrt{w}} \nabla w$ is a product of two $L^2$ functions and hence integrable.
Moreover, as discussed above, the first factor (of the corresponding product for the approximate equations) converges strongly in $L^2$ so that the product converges weakly in $L^1$.
Since the other terms in the weak formulation of \eqref{P} can be seen to converge to their appropriate counterparts as well,
we are finally able to prove Theorem~\ref{TH;3D} in Lemma~\ref{lm:uvwz_weak_sol}.

\section{Preliminaries}
In this section, we collect some general inequalities used multiple times in the sequel.
We start with stating a consequence of the Gagliardo--Nirenberg inequality.
\begin{lem}\label{lm:space_time_interpol}
  Let $n \in \N$, $\Omega \subset \R^n$ be a smooth, bounded domain, $\lambda \in (0, 1)$, $p > 0$ and $q \ge 1$.
  For all $K > 0$, we can then find $C > 0$ such that all nonnegative $\varphi \in \leb p$ with $\varphi^\lambda \in \sob1q$ and
  \begin{align}\label{eq:space_time_interpol:ass}
    \|\varphi\|_{\leb p} \le K
  \end{align}
  satisfy
  \begin{align}\label{eq:space_time_interpol:est}
    \intom \varphi^{\frac{q(\lambda n + p)}{n}} + \intom |\nabla \varphi|^\frac{q(\lambda n + p)}{n + p} \le C \intom |\nabla \varphi^\lambda|^q + C
  \end{align}
  In particular, for all $K' > 0$, there is $C' > 0$ such that all nonnegative functions $φ$ with $\sqrt{\varphi} \in \sob12$ and $\intom \varphi \le K'$ fulfil
  \begin{align}\label{eq:space_time_interpol:special}
    \intom \varphi^{\frac{n+2}{n}} + \intom |\nabla \varphi|^\frac{n+2}{n+1} \le C' \intom \frac{|\nabla \varphi|^2}{\varphi} + C'
  \end{align}
\end{lem}
\begin{remark}
  As the proof will show, the first summand on the left of \eqref{eq:space_time_interpol:est} can be estimated in the same way also if $λ\ge1$.
\end{remark}
\begin{proof}
  If $s \defs \frac{q(\lambda n + p)}{n} \le p$, the first term in \eqref{eq:space_time_interpol:est} is directly estimated by \eqref{eq:space_time_interpol:ass}.
  Else, we set $\theta \defs \frac{\lambda q}{s} = \frac{\lambda n}{\lambda n + p} \in (0, 1)$.
  Then
  \begin{align*}
      \left(1 - \frac{n}{q}\right) \theta + \left(0 - \frac{\lambda n}{p} \right) (1 - \theta)
    = \frac{\lambda n}{\lambda n + p} - \frac{n}{q} \cdot \frac{\lambda q}{s} - \frac{\lambda n}{p} \cdot \frac{p}{\lambda n + p}
    = 0 - \frac{n\lambda}{s},
  \end{align*}
  so that the Gagliardo--Nirenberg inequality
  (which due to $\frac{s}{\lambda} \ge \frac{p}{\lambda}$ and $q\ge 1$ is applicable even if $\f s{λ}<1$ or $\f{q}{λ}<1$, see \cite[Lemma~2.3]{LiLankeitBoundednessChemotaxisHaptotaxis2016})
  asserts that there is $c_1 > 0$ such that
  \begin{align*}
        \|\varphi\|_{\leb s}^s
    =   \|\varphi^\lambda \|_{\leb{\frac{s}{\lambda}}}^\frac{s}{\lambda}
    \le c_1 \|\varphi^\lambda\|_{\sob1q}^\frac{s\theta}{\lambda} \|\varphi^\lambda\|_{\leb{\frac{p}{\lambda}}}^\frac{s(1-\theta)}{\lambda}
    =   c_1 \|\varphi^\lambda\|_{\sob1q}^q \|\varphi\|_{\leb p}^{s(1-\theta)}
  \end{align*}
  for all considered $\varphi$. That is, the first summand in \eqref{eq:space_time_interpol:est} can be appropriately estimated.
   
  The restriction $\lambda \in (0, 1)$ entails that $r \defs \frac{q(\lambda n + p)}{n + p} < q$, so that Young's inequality implies
  \begin{align*}
        |\nabla \varphi|^r
    =   \lambda^{-r} \varphi^{(1-\lambda)r} |\nabla \varphi^\lambda|^r
    \le \lambda^{-\frac{qr}{q-r}} \varphi^{\frac{(1-\lambda)qr}{q-r}} + |\nabla \varphi^\lambda|^q
  \end{align*}
  for all considered $\varphi$,
  where  
  \begin{align*}
      \frac{(1-\lambda)qr}{q-r}
    = \frac{1-\lambda}{\frac1r-\frac1q}
    = \frac{1-\lambda}{\frac{n + p}{q(\lambda n + p)} - \frac1q}
    = \frac{(1-\lambda)q(\lambda n + p)}{n - \lambda n}
    = \frac{q(\lambda n + p)}{n}.
  \end{align*}
  Therefore, we also obtain the estimate for the remaining term in \eqref{eq:space_time_interpol:est}.

  Finally, \eqref{eq:space_time_interpol:special} follows from the general case by setting $\lambda = \frac12$, $p = 1$ and $q = 2$.
\end{proof}

Next, we recall several inequalities regarding $\intom \varphi |D^2 \ln \varphi|^2$.
\begin{lem}\label{lm:nabla_varphi_l14}
  Let $n \in \N$ and $\Omega \subset \R^n$ be a smooth, bounded domain.
  Then
  \begin{align}\label{eq:nabla_varphi_l14:est1}
          \intom \frac{|\nabla \varphi|^4}{\varphi^3}
    &\le  (2 + \sqrt n)^2 \intom \varphi |D^2 \ln \varphi|^2
  \intertext{and}\label{eq:nabla_varphi_l14:est2}
          \intom |D^2 \sqrt \varphi|^2
    &\le  \left(1 + \frac{\sqrt n}{2} + \frac{n}{8}\right) \intom \varphi |D^2 \ln \varphi|^2
  \end{align}
  for all positive $\varphi\in C^2(\ol{\Omega})$ with $\partial_\nu \varphi = 0$ on $\partial \Omega$.
\end{lem}
\begin{proof}
  The first estimate is a special case of \cite[Lemma~3.3]{WinklerGlobalLargedataSolutions2012}, the second one is proven in \cite[Lemma~A.1]{CieslakEtAlExistenceGlobalSolutions2024} (see also \cite[(2.13)]{JiangEtAlGlobalExistenceAsymptotic2015}).
\end{proof}

\begin{lem}\label{lem;GeneralLemma}
  Let $n \in \N$ and $\Omega \subset \R^n$ be a smooth, bounded domain.
  Then there is $C > 0$ such that for all positive $\varphi\in C^2(\ol{\Omega})$ satisfying $\partial_\nu \varphi =0$ on $\partial\Omega$, it holds that
  \begin{align}\label{eq:GeneralLemma:est2}
        - 2\io \frac{|\Delta \varphi|^2}{\varphi}
        + \io \frac{|\nabla \varphi|^2}{\varphi^2}\Delta \varphi
    \le - \io \varphi |D^2 \ln \varphi|^2
        + C \intom \varphi.
  \end{align}
\end{lem}
\begin{proof}
  Specializing the identities in \cite[(3.5) and the one at the top of page~331]{WinklerGlobalLargedataSolutions2012} to $g(\varphi) \defs \varphi$
  (which implies $\rho(\varphi) = \ln \varphi$ by \cite[(3.2)]{WinklerGlobalLargedataSolutions2012})
  shows
  \begin{align}\label{equal;Deltaphi^2/phi}
      - \io \frac{|\Delta \varphi|^2}{\varphi}
    = - \io \varphi |D^2 \ln \varphi|^2
      - \frac 12 \io \frac{|\nabla \varphi|^2}{\varphi^2}\Delta \varphi
      + \frac 12 \int_{\partial \Omega} \frac 1\varphi \partial_\nu |\nabla \varphi|^2
  \end{align}
  for all positive $\varphi\in C^2(\ol{\Omega})$ with $\partial_\nu \varphi =0$ on $\partial\Omega$.
  By \cite[(2.8)--(2.10)]{JiangEtAlGlobalExistenceAsymptotic2015}, there is $c_1 > 0$ such that
  \begin{align}\label{eq:GeneralLemma:bdry_est}
    \int_{\partial\Omega} \frac 1\varphi \partial_\nu |\nabla \varphi|^2 
    \le c_1 \left( \lp{2}{\Delta \sqrt{\varphi}}^{\frac{3+1/4}{2}} \lp{2}{\sqrt{\varphi}}^{\frac{1-1/4}{2}} + \lp{2}{\sqrt{\varphi}}^2 \right) 
  \end{align}
  holds for all such $\varphi$.
  We set $c_2 \defs (1 + \frac{\sqrt n}{2} + \frac{n}{8})$
  and note that Young's inequality asserts $ab \le \frac1{c_1c_2} a^\frac{4}{3+1/4} + c_3 b^\frac{4}{1-1/4}$ for some $c_3 > 0$ and all $a, b > 0$.
  Thus, combining \eqref{eq:GeneralLemma:bdry_est} and \eqref{eq:nabla_varphi_l14:est2} warrants that
  \begin{align*}
    \f12\int_{\partial\Omega} \frac 1\varphi \partial_\nu |\nabla \varphi|^2 
    \le \frac{1}{2c_2} \lp{2}{\Delta \sqrt{\varphi}}^2 + \f{c_1c_3 + c_1}{2} \lp{2}{\sqrt{\varphi}}^2
    \le \f12\io \varphi |D^2 \ln \varphi|^2 + \f{c_1 c_3+c_1}{2} \intom \varphi
  \end{align*}
  for all positive $\varphi\in C^2(\ol{\Omega})$ with $\partial_\nu \varphi =0$ on $\partial\Omega$,
  which yields \eqref{eq:GeneralLemma:est2} when plugged into \eqref{equal;Deltaphi^2/phi}.
\end{proof}

\section{Quasi-energy functionals for subsystems and a~priori estimates}\label{sec:combined}
In the present section, we consider
\begin{align}\label{P;sigma}
  \begin{cases}
    u_t = \Du \Delta u - \chiu \nabla \cdot (u \sigma'(u) \nabla v) - \gammau \sigma(u) v - \deltau u + \betau & \text{in $\Omega \times (0, \infty)$}, \\
    v_t = \Dv \Delta v + \rhov v - \gammav \sigma(u) v + \muv w                                                 & \text{in $\Omega \times (0, \infty)$}, \\
    w_t = \Dw \Delta w + \gammaw \sigma(u)v - \alphaw w \sigma(z) - \muw w                                      & \text{in $\Omega \times (0, \infty)$}, \\
    z_t = \Dz \Delta z - \chiz \nabla \cdot (z \sigma'(z) \nabla w) + \alphaz f(w)\sigma(z) - \deltaz z         & \text{in $\Omega \times (0, \infty)$}, \\
    \partial_\nu u = \partial_\nu v = \partial_\nu w = \partial_\nu z = 0                                       & \text{on $\partial \Omega \times (0, \infty)$}, \\
    (u, v, w, z)(\cdot, 0) = (u_0, v_0, w_0, z_0)                                                               & \text{in $\Omega$},
  \end{cases}
\end{align}
where $(u_0, v_0, w_0, z_0) \in \con0 \times \sob1q \times \sob1q \times \con0$ for some $q > n$ is nonnegative and such that
\begin{align}\label{eq:def_A}
    \intom u_0 \ln u_0 + \|v_0\|_{\leb\infty} + \|w_0\|_{\leb3} + \intom z_0 \ln z_0
  + \intom \frac{|\nabla v_0|^2}{v_0} + \intom \frac{|\nabla w_0|^2}{w_0} \le A
\end{align}
for some $A > 0$,
and where $\sigma$ belongs to
\begin{align}\label{eq:Sigma}
  \Sigmaset \defs \Big\{\tilde \sigma \in C^\infty([0, \infty)) \;\Big\vert\; \tilde \sigma(0) = 0,\, 0 \le \tilde \sigma'(s) \le 1,\, \frac{s\tilde{\sigma}'(s)}{\tilde{\sigma}(s)}\le 2\ \mbox{for every}\ s> 0  \Big\}.
\end{align}
System \eqref{P;sigma} reduces to \eqref{P} for the choice $\sigma = \operatorname{id} \in \Sigmaset$, which we shall take in Section~\ref{sec:2d} for proving Theorem~\ref{TH;2D}.
In order to construct weak solutions in three dimensional settings, and thus prove Theorem~\ref{TH;3D},
we will choose a family $(\sigma_\eps)_{\eps \in (0, 1)} \subset \Sigmaset$ converging pointwise to $\operatorname{id}$ as $\eps \sea 0$
and being such that the solutions of \eqref{P;sigma} with $\sigma_\eps$ instead of $\sigma$ are easily seen to be global in time, see Section~\ref{sec:3d}.
In order to then derive $\eps$-independent estimates from the results from the present section,
we need to make sure that all constants below are independent of $\sigma$ (as long as $\sigma \in \Sigmaset$) and that they depend on the initial data only through $A$.

We first state existence of local classical solutions of \eqref{P;sigma}, along with a criterion for when these are global.
\begin{lem}\label{lm:local_ex}
  Let $n \in \N$, let $\Omega \subset \R^n$ be a smooth, bounded domain, assume \eqref{eq:intro:params} and \eqref{cond:f}, let $\sigma \in \Sigmaset$
  and suppose that $(u_0, v_0, w_0, z_0) \in \con0 \times \sob1q \times \sob1q \times \con0$ for some $q > n$ is nonnegative.
  Then there exist $\tmax \in (0, \infty)$ and a uniquely determined solution
  \begin{align}\label{eq:local_ex:reg}
    (u, v, w, z) \in C^0\big([0, \tmax); \con0 \times \sob1q \times \sob1q \times \con0\big) \cap \big(C^\infty(\Ombar \times (0, \tmax))\big)^4
  \end{align}
  of \eqref{P;sigma} which is nonnegative (and even positive if the initial data are nontrivial) and has the property that
  \begin{align}\label{eq:local_ex:ext_crit}
    \tmax < \infty
    \quad \text{implies} \quad
    \limsup_{t \nea \tmax} \left( \|u(\cdot, t)\|_{\leb \infty} + \|z(\cdot, t)\|_{\leb \infty} \right) = \infty.
  \end{align}
  Furthermore, if $(u_0, v_0, w_0, z_0) \in \big(\con\infty\big)^4$, then $(u, v, w, z) \in \big(C^\infty(\Ombar \times [0, \tmax))\big)^4$.
\end{lem}
\begin{proof}
  The existence of a unique local classical solution of regularity \eqref{eq:local_ex:reg} satisfying the extensibility criterion \eqref{eq:local_ex:ext_crit}
  can be shown by a fixed point argument (see for instance \cite[Lemma~3.1 and Lemma~3.2]{BellomoEtAlMathematicalTheoryKeller2015}).
  Nonnegativity (and positivity) follow from the (strict) maximum principle (cf.\ \cite[Proposition~52.7]{QuittnerSoupletSuperlinearParabolicProblems2019})
  and further smoothness properties are a consequence of, e.g., \cite[Corollary~14.7]{AmannNonhomogeneousLinearQuasilinear1993}.
\end{proof}

We henceforth fix
\begin{align}\label{eq:ass_23}
  \text{$n \in \{2, 3\}$, a smooth, bounded domain $\Omega \subset \R^n$, parameters as in \eqref{eq:intro:params}, $f$ as in \eqref{cond:f}, $A > 0$},
\end{align}
and shall obtain uniform estimates for
\begin{align}\label{eq:sol_class}
  \begin{cases}
    \text{solutions $(u, v, w, z)$ of \eqref{P;sigma} given by Lemma~\ref{lm:local_ex} with maximal existence time $\tmax$} \\
    \text{for some $\sigma \in \Sigmaset$ and $(u_0, v_0, w_0, z_0) \in \big(C^\infty(\Ombar; [0, \infty))\big)^4$ satisfying \eqref{eq:def_A}}.
  \end{cases}
\end{align}

The first yet very basic a~priori estimates for solutions of \eqref{P;sigma} are the following. 
\begin{lem}\label{lem;L1Estimate}
  Assume \eqref{eq:ass_23} and let $T \in (0, \infty)$.
  Then there is $C>0$ such that for all $(u, v, w, z)$ as in \eqref{eq:sol_class} with $\tmax \ge T$,
  \begin{align}\label{eq:L1Estimate:est}
    \sup_{t \in (0, T)} \left( \intom u +  \intom v +  \intom w + \intom z \right) \le C.
  \end{align} 
\end{lem}
\begin{proof}
  Integrating a suitable linear combination of the parabolic equations in \eqref{P;sigma} shows that
  \begin{align*}
    y \defs \io u + \io v+ \frac{\gammau + \gammav}{\gammaw} \io w + \frac{\alphaw (\gammau + \gammav)}{\alphaz\gammaw} \io z
  \end{align*}
  fulfils
  \begin{align*}
          \ddt y
    &\le  - \deltau \io u + \betau |\Omega|
          + \rhov \io v
          + \left( \muv - \frac{\muw (\gammau+\gammav)}{\gammaw} \right)\io w - \frac{\alphaw \deltaz (\gammau + \gammav)}{\alphaz\gammaw} \io z \\
    &\le  \max\left\{\rhov, \frac{\muv \gammaw}{\gammau + \gammav}\right\} y + \betau |\Omega|
  \end{align*}
  in $(0, \tmax)$ for all $(u, v, w, z)$ as in \eqref{eq:sol_class}.
  The statement then follows by an ODE comparison argument and the $L^1$ bounds for the initial data implicitly contained in \eqref{eq:def_A}.
\end{proof}

While for chemotaxis--consumption systems such as \eqref{prob:ct_con}, the signal is easily seen to be bounded, this is no longer the case for \eqref{P;sigma}.
However, since stronger bounds for $v$ imply stronger bounds for $w$ and vice-versa (up to a certain point at least),
the estimates for $v$ and $w$ in \eqref{eq:L1Estimate:est} can be bootstrapped in the following way.
\begin{lem}\label{lem;Linf-v}
  Assume \eqref{eq:ass_23}, let $T \in (0, \infty)$ and let $p \in [1, \frac{n}{n-2})$.
  Then there is $C>0$ such that for all $(u, v, w, z)$ as in \eqref{eq:sol_class} with $\tmax \ge T$,
  \begin{align}\label{eq:Linf-v:est}
    \sup_{t \in (0, T)} \left( \|v(\cdot, t)\|_{\leb \infty} + \|w(\cdot, t)\|_{\leb p} \right) \le C.
  \end{align}
\end{lem}
\begin{proof}
  Since $n \le 3$, we may without loss of generality assume $p > \frac n2$.
  We set
  \begin{align*}
    M_v(T') \defs \sup_{t \in (0, T')} \|v(\cdot, t)\|_{\leb\infty}, \quad
    M_w(T') \defs \sup_{t \in (0, T')} \|w(\cdot, t)\|_{\leb p}
  \end{align*}
  for $(u, v, w, z)$ as in \eqref{eq:sol_class} with $\tmax \ge T$ and every $T' \in (0, T)$,
  fix $\lambda \in (\frac n2, p)$ and set $p' \defs \frac{p}{p-1} > \frac n2$.
  By Duhamel's formula, the maximum principle, semigroup estimates (cf.\ \cite[Lemma~1.3]{WinklerAggregationVsGlobal2010}), Hölder's inequality, \eqref{eq:def_A} and \eqref{eq:L1Estimate:est},
  there are $c_1, \dots, c_7 > 0$ such that, with $\theta = \f{(λ-1)p}{λ(p-1)}\in (0,1)$,
  \begin{align}\label{eq:Linf-v:vinfty}
          \|v(\cdot, t)\|_{\leb\infty} 
    &\le  \|\ure^{t(\Dv\Delta+\rhov)} v_0\|_{\leb \infty}
          + \muv \int_0^t \|\ure^{(t-s)(\Dv\Delta+\rhov)} w(\cdot, s)\|_{\leb \infty} \ds \notag \\
    &\le  \ure^{\rhov T}\|v_0\|_{\leb \infty} 
          + c_1 \int_0^t (1 + (t-s)^{-\frac n2 \cdot (\frac1\lambda-\frac1\infty)}) \ure^{\rhov (t-s)} \|w(\cdot, s)\|_{\leb{\lambda}} \ds \notag \\
    &\le  c_2 
          + c_2 \ure^{\rhov T} \sup_{s \in (0, T')} \|w(\cdot, s)\|_{\leb{\lambda}}  \int_0^T (1 + s^{-\frac{n}{2\lambda}}) \ds \notag \\
    &\le  c_2 
          + c_3 \sup_{s \in (0, T')} \left( \|w(\cdot, s)\|_{\leb p}^{\theta} \|w(\cdot, s)\|_{\leb1}^{1-\theta} \right)
     \le  c_2 
          + c_4 \big(M_w(T')\big)^\theta
  \end{align}
  and
  \begin{align}\label{eq:Linf-v:wp}
          \|w(\cdot, t)\|_{\leb p} 
    &\le  \|\ure^{t \Dw \Delta} w_0\|_{\leb p}
          + \gammaw \int_0^t \|\ure^{(t-s) \Dw \Delta} (uv)(\cdot, s)\|_{\leb p} \ds \notag \\
    &\le  c_5 \|w_0\|_{\leb p} 
          + c_5 \int_0^t (1 + (t-s)^{-\frac n2(\frac11 - \frac1p)}) \|(uv)(\cdot, s)\|_{\leb1} \ds \notag\\
    &\le  c_6 \|w_0\|_{\leb 3} 
          + c_6 \sup_{s \in (0, T')} \left( \|v(\cdot, s)\|_{\leb\infty} \|u(\cdot, s)\|_{\leb1} \right) \int_0^T (1 + s^{-\frac{n}{2p'}}) \ds \notag \\
    &\le  c_7
          + c_7 M_v(T')
  \end{align}
  for all $t \in (0, T')$, all $(u, v, w, z)$ as in \eqref{eq:sol_class} with $\tmax \ge T$ and all $T' \in (0, T)$.
  By Young's inequality and due to $\theta \in (0, 1)$, there is $c_8 > 0$ such that $c_4 c_7^\theta a^\theta \le c_8 + \frac{a}{2}$ for all $a \ge 0$.
  Thus, by taking the supremum over $(0, T')$ in each of the estimates \eqref{eq:Linf-v:vinfty} and \eqref{eq:Linf-v:wp}, we conclude
  \begin{align*}
        M_v(T')
    \le c_2 + c_4 \big(M_w(T')\big)^\theta
    \le c_2 + c_4c_7^\theta + c_4 c_7^\theta \big(M_v(T')\big)^\theta
    \le c_2 + c_4c_7^\theta + c_8 + \frac{M_v(T')}{2}
  \end{align*}
  and hence $M_v(T') \le c_9 \defs 2(c_2 + c_4c_7^\theta + c_8)$ as well as $M_w(T') \le c_7 + c_7c_9$
  for all $(u, v, w, z)$ as in \eqref{eq:sol_class} with $\tmax \ge T$ and all $T' \in (0, T)$,
  so that \eqref{eq:Linf-v:est} follows upon taking $T' \nea T$.
\end{proof}

Next, we consider adaptations of the energy functional discovered in \cite{DuanEtAlGlobalSolutionsCoupled2010} for chemotaxis-consumption systems to both the $(u, v)$- and the $(z, w)$-subsystem.
Mainly due to the production terms in the signal equations in \eqref{P;sigma}, however, several modifications are necessary.
\begin{lem}\label{lem;firstEn}
  Assume \eqref{eq:ass_23} and let $T \in (0, \infty)$.
  Then there is $C>0$ such that for all $(u, v, w, z)$ as in \eqref{eq:sol_class} with $\tmax \ge T$, and with $σ$ from \eqref{eq:sol_class},
  \begin{align}\label{eq:firstEn:space}
   \io u\ln u + \io \frac{|\nabla v|^2}{v} + \io w \ln w \le C
   \qquad \text{in $(0, T)$}
  \end{align}
  and 
  \begin{align}\label{eq:firstEn:spacetime}
    \iio \frac{|\nabla u|^2}{u} + \iio v |D^2 \ln v|^2 + \iio \frac{|\nabla v|^2}{v}\sigma(u) \le C.
  \end{align}
\end{lem}
\begin{proof}
  Straightforward testing procedures yield that
  \begin{align}\label{eq:firstEn:u}
        \ddt \io (u\ln u - u)
    &=  - \Du \io \frac{|\nabla u|^2}{u}
        + \chiu \io \sigma'(u)\nabla u \cdot \nabla v \notag \\
    &\pe -\, \gammau \io \sigma(u) v \ln u 
        - \deltau \io u \ln u + \betau \io \ln u
  \end{align}
  and
  \begin{align}\label{eq:firstEn:w}
        \ddt \io (w\ln w - w)
    &=  -\, \Dw \io \frac{|\nabla w|^2}{w}
        + \gammaw\io\sigma(u)v\ln w
        - \alphaw \io w\sigma(z)\ln w
        - \muw\io w\ln w
  \end{align}
  as well as
  \begin{align*}
          \ddt \io \frac{|\nabla v|^2}{v} 
    &=    - 2 \io \frac{\Delta v}{v} v_t + \io \frac{|\nabla v|^2}{v^2} v_t \\
    &=    - 2 \Dv \io \frac{|\Delta v|^2}{v} 
          + \Dv \io \frac{|\nabla v|^2}{v^2}\Delta v
          + 2 \gammav \io \sigma(u) \Delta v 
          - \gammav \io \frac{|\nabla v|^2}{v}\sigma(u) \\
    &\pe  - 2\rhov \io \Delta v
          + \rhov \io \frac{|\nabla v|^2}{v} 
          - 2\muv \io \frac{\Delta v \cdot w}{v}
          + \muv \io \frac{|\nabla v|^2}{v^2}w \\
    &=    - 2 \Dv \io \frac{|\Delta v|^2}{v} 
          + \Dv \io \frac{|\nabla v|^2}{v^2}\Delta v
          - 2 \gammav \io \sigma'(u) \nabla u \cdot \nabla v 
          - \gammav \io \frac{|\nabla v|^2}{v}\sigma(u) \\
    &\pe  +\, \rhov \io \frac{|\nabla v|^2}{v} 
          + 2\muv \io \frac{\nabla v\cdot \nabla w}{v} 
          - \muv \io \frac{|\nabla v|^2}{v^2}w
  \end{align*}
  hold in $(0, \tmax)$ for all $(u, v, w, z)$ as in \eqref{eq:sol_class}.
  By Young's inequality, we have 
  \[
    2\muv \io \frac{\nabla v\cdot \nabla w}{v} \le \muv \io \frac{|\nabla v|^2}{v^2} w + \muv \io \frac{|\nabla w|^2}{w},
  \]
  and due to Lemma~\ref{lem;GeneralLemma} and Lemma~\ref{lem;L1Estimate}, we can find $c_1>0$ such that 
  \[
    -2\Dv \io \frac{|\Delta v|^2}{v} + \Dv  \io \frac{|\nabla v|^2}{v^2} \Delta v \le - \Dv  \io v|D^2 \ln v|^2  + c_1
  \]
  in $(0, T)$ for all $(u, v, w, z)$ as in \eqref{eq:sol_class} with $\tmax \ge T$,
  which implies that 
  \begin{align}\label{eq:firstEn:v}
          \ddt \io \frac{|\nabla v|^2}{v}
    &\le  - \Dv  \io v|D^2 \ln v|^2  
          - 2 \gammav \io \sigma'(u) \nabla u \cdot \nabla v
          - \gammav \io \frac{|\nabla v|^2}{v}\sigma(u) \notag \\
    &\pe  +\, \rhov \io \frac{|\nabla v|^2}{v} 
          + \muv\io \frac{|\nabla w|^2}{w}
          + c_1
  \end{align}
  in $(0, T)$ for all $(u, v, w, z)$ as in \eqref{eq:sol_class} with $\tmax \ge T$.
   
  Combining \eqref{eq:firstEn:u}, \eqref{eq:firstEn:w} and \eqref{eq:firstEn:v} shows that
  \begin{align}\label{eq:firstEn:def_y}
    y \defs 
    \io (u\ln u-u) + \frac{\chiu}{2\gammav} \io \frac{|\nabla v|^2}{v} 
    + \frac{\muv \chiu}{2\Dw \gammav} \io (w\ln w-w)
  \end{align}
  fulfils
  \begin{align*}
    y'
  &\le - \Du \io \frac{|\nabla u|^2}{u} 
   -\frac{\chiu \Dv}{2\gammav} \io v|D^2 \ln v|^2 
   -\frac{\chiu}{2} \io \frac{|\nabla v|^2}{v} \sigma(u) 
   + \frac{\chiu \rhov}{2\gammav}\io \frac{|\nabla v|^2}{v}
  \notag \\ &\pe -\, \gammau \io \sigma(u) v \ln u 
   - \deltau \io u \ln u + \betau \io \ln u
  \notag\\ &\pe +\, \frac{\muv \chiu \gammaw}{2\Dw \gammav} \io\sigma(u)v\ln w
   - \frac{\alphaw \muv \chiu}{2\Dw \gammav} \io w\sigma(z)\ln w
   - \frac{\muv \muw \chiu}{2\Dw \gammav} \io w\ln w
   + \f{\chiu c_1}{2\gammav}
  \end{align*}
  in $(0, T)$ if $(u, v, w, z)$ are as in \eqref{eq:sol_class} with $\tmax \ge T$.
   
  We denote the constants in \eqref{eq:space_time_interpol:special}, \eqref{eq:L1Estimate:est} and \eqref{eq:Linf-v:est} by $c_2$, $c_3$ and $c_4$, respectively.
  Recalling \eqref{eq:Sigma}, we see that 
  \begin{align*}
    - \gammau \io \sigma(u) v \ln u + \betau \io \ln u &\le \frac{\gammau}{\ure} \io v + \betau \io u \le \frac{(\gammau + \betau \ure)c_3}{\ure}, \qquad 
    - \deltau \io u \ln u \le \frac{\deltau |\Omega|}{\ure}, \\
    - \frac{\alphaw \muv \chiu}{2\Dw \gammav} \io w \sigma(z) \ln w &\le \frac{\alphaw \muv \chiu}{2\Dw \gammav \ure} \io z \le \frac{\alphaw \muv \chiu c_3 }{2\Dw \gammav \ure}, \qquad 
    - \frac{\muv \muw \chiu}{2\Dw \gammav} \io w\ln w \le \frac{\muv \muw \chiu|\Omega|}{2\Dw \gammav\ure},
  \end{align*}
  and that there is moreover $c_5 > 0$ such that with $c_6 \defs \frac{\Du}{2} + c_3 c_4 c_5$,
  \begin{align*}
  \frac{\muv \chiu \gammaw}{2\Dw \gammav} \io \sigma(u) v \ln w 
  &\le c_4 \io uw^{\frac 2{n+2}}
  \le c_4 \left(\frac{\Du}{2c_2c_4} \io u^{\frac{n+2}n} + c_5 \io w  \right)
  \le \frac{\Du}2 \io \frac{|\nabla u|^2}{u} + c_6
  \end{align*}
  in $(0, T)$ for all $(u, v, w, z)$ as in \eqref{eq:sol_class} with $\tmax \ge T$.

  Therefore, there are $c_7,c_8 > 0$ such that the function $y$ defined in \eqref{eq:firstEn:def_y} which, due to \eqref{eq:L1Estimate:est}, is bounded from below satisfies
  \begin{align}\label{eq:firstEn:y2}
  &y'
  + \frac{\Du}{2} \io \frac{|\nabla u|^2}{u} 
  + \frac{\chiu \Dv}{2\gammav} \io v|D^2 \ln v|^2 
  + \frac{\chiu}{2} \io \frac{|\nabla v|^2}{v} \sigma(u) 
  \le c_7 + c_8 y
  \end{align}
  in $(0, T)$, whenever $(u, v, w, z)$ are as in \eqref{eq:sol_class} with $\tmax \ge T$.
  Since \eqref{eq:def_A} entails that there is $c_9 > 0$ such that $y(0) \le c_9$ for all $(u, v, w, z)$ as in \eqref{eq:sol_class},
  and again due to the $L^1$ bounds in \eqref{eq:L1Estimate:est},
  integrating \eqref{eq:firstEn:y2} yields \eqref{eq:firstEn:space} and \eqref{eq:firstEn:spacetime} for some $C > 0$.
\end{proof}

\begin{lem}\label{lem;secondEn}
  Assume \eqref{eq:ass_23} and let $T \in (0, \infty)$.
  Then there is $C>0$ such that for all $(u, v, w, z)$ as in \eqref{eq:sol_class} with $\tmax \ge T$, and with $σ$ from \eqref{eq:sol_class},
  \begin{align}\label{eq:secondEn:space}
    \io z\ln z + \io \frac{|\nabla w|^2}{w} \le C
    \qquad \text{in $(0, T)$}
  \end{align}
  and
  \begin{align}\label{eq:secondEn:spacetime}
    \iio \frac{|\nabla z|^2}{z} + \iio w|D^2 \ln w|^2 + \iio \frac{|\nabla w|^2}{w} \sigma(z) \le C.
  \end{align}
\end{lem} 
\begin{proof}
  Similarly as in the proof of Lemma~\ref{lem;firstEn}, straightforward calculations yield
  \begin{align*}
        \ddt \io (z\ln z-z)
    &=  -\Dz \io \frac{|\nabla z|^2}{z}
        + \chiz \io \sigma'(z)\nabla z\cdot \nabla w
        + \alphaz \io f(w) \sigma(z)\ln z - \deltaz \io z\ln z
  \end{align*}
  and 
  \begin{align*}
        \ddt \io \frac{|\nabla w|^2}{w}
    &=  - 2 \Dw \io \frac{|\Delta w|^2}{w}
        + \Dw \io \frac{|\nabla w|^2}{w^2} \Delta w
        - \muw \io \frac{|\nabla w|^2}{w} 
    \\ &\pe
        +\, 2\gammaw \io \frac{\nabla w\cdot \nabla (\sigma (u)v)}{w} 
        - \gammaw \io \frac{|\nabla w|^2}{w} \sigma (u) v
    \\ &\pe
        -\, 2\alphaw \io \sigma'(z)\nabla z \cdot \nabla w  
        - \alphaw \io \frac{|\nabla w|^2}{w} \sigma(z) 
  \end{align*}
  in $(0, \tmax)$ for all $(u, v, w, z)$ as in \eqref{eq:sol_class}.
  Again making use of Lemma~\ref{lem;GeneralLemma}, we see that there is $c_1 > 0$ such that
  \begin{align*}
    - 2 \Dw \io \frac{|\Delta w|^2}{w} + \Dw \io \frac{|\nabla w|^2}{w^2} \Delta w \le - \Dw \io w |D^2\ln w|^2  + c_1
  \end{align*}
  in $(0, T)$ for all $(u, v, w, z)$ as in \eqref{eq:sol_class} with $\tmax \ge T$.
     
  Moreover, since $\frac{(\sigma'(s))^2}{\sigma(s)} \le \frac{2}{s}$ holds for all $s > 0$ by \eqref{eq:Sigma},
  we may further estimate
  \begin{align*}
    2\gammaw \io \frac{\nabla w \cdot \nabla (\sigma(u)v)}{w} 
    & \le 
    \gammaw \io \frac{|\nabla w|^2}{w^2} \sigma(u)v + \gammaw \io \frac{|\nabla (\sigma(u)v)|^2}{\sigma(u)v}
  \\
  & \le
   \gammaw \io \frac{|\nabla w|^2}{w^2} \sigma(u)v + 2\gammaw \io \frac{(\sigma'(u))^2v}{\sigma(u)} |\nabla u|^2 + 2\gammaw \io \frac{|\nabla v|^2}{v} \sigma(u)
  \\ & \le 
   \gammaw \io \frac{|\nabla w|^2}{w^2} \sigma(u)v + 4\gammaw \|v\|_{L^\infty(\Omega \times (0, T))} \io \frac{|\nabla u|^2}{u} + 2\gammaw \io \frac{|\nabla v|^2}{v} \sigma(u)
  \end{align*}
  in $(0, T)$ for all $(u, v, w, z)$ as in \eqref{eq:sol_class} with $\tmax \ge T$.
    
  Let $c_2 > 0$ be the constant in \eqref{eq:space_time_interpol:special}.
  Then \eqref{eq:Sigma}, \eqref{eq:space_time_interpol:special} and \eqref{eq:Linf-v:est} (noting that $\frac{3n+5}{5} < \frac{n}{n-2}$ since $n \in \{2, 3\}$) allow us to find $c_3, c_4, c_5 > 0$ such that
  \begin{align*}
         \frac{\alphaz \chiz}{2\alphaw} \io f(w)\sigma(z)\ln z 
    &\le c_3 \io wz^{\frac{3(n+2)}{3n+5}}
     \le \frac{\Dz}{2c_2} \io z^{\frac{n+2}n} + c_4 \io w^{\frac{3n+5}{5}}
     \le \frac{\Dz}{2} \io \frac{|\nabla z|^2}{z} + c_5
  \end{align*}
  in $(0, T)$ for all $(u, v, w, z)$ as in \eqref{eq:sol_class} with $\tmax \ge T$.

  Combining these estimates and recalling that $v$ is bounded in $\Omega \times (0, T)$ by \eqref{eq:Linf-v:est},
  we see that
  \begin{align*}
    y \defs \io (z\ln z-z) + \frac{\chiz}{2\alphaw} \io \frac{|\nabla w|^2}{w}
  \end{align*}
  fulfils
  \begin{align*}
    y'
  & \le 
    - \frac{\Dz}{2} \io \frac{|\nabla z|^2}{z}  
    - \frac{\chiz\Dw}{2\alphaw} \io w|D^2\ln w|^2 
    - \frac{\chiz}{2} \io \frac{|\nabla w|^2}{w} \sigma(z)
  \\ &\pe
    +\, \frac{2\chiz \gammaw c_6}{\alphaw} \io \frac{|\nabla u|^2}{u} 
    + \frac{\chiz\gammaw}{\alphaw} \io \frac{|\nabla v|^2}{v} \sigma(u)
    + \frac{\deltaz|\Omega|}{\ure}
    + \frac{\chiz c_1}{2\alphaw}
    + c_5
  \end{align*}
  in $(0, T)$, whenever $(u, v, w, z)$ is as in \eqref{eq:sol_class} with $\tmax \ge T$.
  Since $y(0)$ is bounded by \eqref{eq:sol_class} and due to \eqref{eq:firstEn:spacetime},
  an ODE comparison argument proves \eqref{eq:secondEn:space} and \eqref{eq:secondEn:spacetime} for some $C > 0$.
\end{proof}

\section{Global classical solutions in 2D}\label{sec:2d}
In two-dimensional settings, solutions to the Keller--Segel system with linear diffusion and linear taxis sensitivity are global in time whenever the chemotactically active component is bounded in $\lebl$,
see for instance \cite[Lemma~3.3]{BellomoEtAlMathematicalTheoryKeller2015}.
The aim of the present section is to show that also an analogous result holds for the more complex system \eqref{P};
that is, we shall bootstrap the a~priori estimates from Section~\ref{sec:combined}, mainly Lemma~\ref{lem;firstEn} and Lemma~\ref{lem;secondEn}, to $L^\infty$ bounds
and thereby prove Theorem~\ref{TH;2D}. 

Without commenting on this any further, we thus assume that $\Omega \subset \R^2$ is a smooth bounded domain,
and fix parameters as in \eqref{eq:intro:params} and $f$ as in \eqref{cond:f} as well as initial data as in \eqref{eq:2d_main:init}.
Since \eqref{P} coincides with \eqref{P;sigma} for $\sigma = \operatorname{id} \in \Sigmaset$,
Lemma~\ref{lm:local_ex} provides us with a (henceforth also fixed) classical solution $(u, v, w, z)$ of \eqref{P} with maximal existence time $\tmax \in (0, \infty]$.
By switching to the solution with initial data $(u, v, w, z)(\cdot, \min\{1, \frac{\tmax}{2}\})$, we may moreover assume that this solution is smooth and positive in $\Ombar \times [0, \tmax)$.
Then \eqref{eq:def_A} is fulfilled for some $A > 0$, so that we may indeed apply the results from Section~\ref{sec:combined} to this solution.

A testing procedure and Ladyzhenskaya's trick (cf.\ \cite{LadyzenskajaSolutionLargeNonstationary1959}) reveal that the a~priori estimates obtained in Lemma~\ref{lem;Linf-v} and Lemma~\ref{lem;firstEn}
imply a uniform $L^2$ bound for the first solutions component.
\begin{lem}\label{lm:u_l2}
  Let $T \in (0, \infty) \cap (0, \tmax]$.
  Then there is $C > 0$ such that
  \begin{align}\label{eq:u_l2:est}
    \io u^2 \le C
    \qquad \text{in $(0, T)$}.
  \end{align}
\end{lem}
\begin{proof}
  By testing the first equation in \eqref{P} with $u$ and dropping some nonpositive terms, we obtain 
  \begin{align*}
        \frac12 \ddt \intom u^2
    \le - \Du \intom |\nabla u|^2
        + \chiu \intom u \nabla u \cdot \nabla v
        + \betau \intom u
    \qquad \text{in $(0, \tmax)$}.
  \end{align*}
  As in the considered two-dimensional setting the Gagliardo--Nirenberg inequality yields $c_1 > 0$ with
  \begin{align*}
    \|\varphi\|_{\leb4}^2 \le c_1 \|\nabla \varphi\|_{\leb2} \|\varphi\|_{\leb2} + c_1 \|\varphi\|_{\leb1}^2
    \qquad \text{for all $\varphi \in \sob12$},
  \end{align*}
  we may estimate
  \begin{align*}
          \intom u \nabla u \cdot \nabla v
    &\le  \frac{\Du}{2} \intom |\nabla u|^2 + \frac{1}{2\Du} \intom u^2 |\nabla v|^2
     \le  \frac{\Du}{2} \intom |\nabla u|^2 + \frac1{2\Du} \|u\|_{\leb4}^2 \|\nabla v\|_{\leb4}^2 \\
    &\le  \frac{\Du}{2} \intom |\nabla u|^2 + \frac{c_1}{2\Du} \left( \|\nabla u\|_{\leb2} \|u\|_{\leb2} + \|u\|_{\leb1}^2 \right) \|\nabla v\|_{\leb4}^2 \\
    &\le  \Du \intom |\nabla u|^2 + \frac{c_1^2}{8\Du^3} \left( \intom u^2 \right) \left( \intom |\nabla v|^4 \right) + \frac{c_1}{2\Du} \|u\|_{\leb1}^2 \|\nabla v\|_{\leb4}^4 + \frac{c_1}{2\Du}\|u\|_{\leb1}^2 \\
    &\le  \Du \intom |\nabla u|^2 + \left(\frac{c_1^2}{8\Du^3}+\frac{c_1|\Omega|}{2\Du}\right) \left( \intom u^2 \right) \left( \intom |\nabla v|^4 \right) + \frac{c_1}{2\Du}\|u\|_{\leb1}^2
    \qquad \text{in $(0, \tmax)$}.
  \end{align*}
  Since $\intom u$ and $\|v\|_{\leb\infty}$ are bounded in $(0, T)$ by Lemma~\ref{lem;L1Estimate} and Lemma~\ref{lem;Linf-v}, we conclude
  \begin{align}\label{eq:u_l2:odi}
        \ddt \intom u^2
    \le c_2 \left( \intom u^2 \right) \left( \intom \frac{|\nabla v|^4}{v^3} \right) + c_2
    \qquad \text{in $(0, T)$}
  \end{align}
  for some $c_2 > 0$.
  Setting $g(t) \defs \intom \big(\frac{|\nabla v|^4}{v^3}\big)(\cdot,t)$ for $t \in (0, T)$, we integrate \eqref{eq:u_l2:odi} to obtain
  \begin{align*}
        \intom u^2(\cdot, t)
    \le \ure^{c_2 \int_0^t g(s) \ds} \intom u_0^2 + c_2 \int_0^t \ure^{c_2 \int_s^t g(\sigma) \dsigma} \ds
    \qquad \text{for all $t \in (0, T)$},
  \end{align*}
  which implies \eqref{eq:u_l2:est} for some $C > 0$,
  since $g$ is integrable in $(0, T)$
  due to \eqref{eq:nabla_varphi_l14:est1} and \eqref{eq:firstEn:spacetime}.
\end{proof}
 
The estimate \eqref{eq:u_l2:est} renders arguments based on semigroup estimates applicable, which allow us to rapidly prove the following.
\begin{lem}\label{lm:u_linfty}
  Let $T \in (0, \infty) \cap (0, \tmax]$.
  Then there is $C > 0$ such that
  \begin{align}\label{eq:u_linfty:est}
    u \le C
    \qquad \text{in $\Omega \times(0, T)$}.
  \end{align}
\end{lem}
\begin{proof}
  Since $g \defs \rhov v - \gammav uv + \muv w \in L^\infty((0, T); \leb2)$ by \eqref{eq:Linf-v:est} and \eqref{eq:u_l2:est},
  semigroup estimates (cf.\ \cite[Lemma~3.1]{WinklerAggregationVsGlobal2010}) yield
  \begin{align}\label{eq:u_linfty:nabla_v_l4}
          \|\nabla v(\cdot, t)\|_{\leb4}
    &\le  \|\nabla \ure^{t \Dv \Delta} v_0\|_{\leb4} 
          + \int_0^t \|\nabla \ure^{(t-s) \Dv \Delta} g(\cdot, s)\|_{\leb4} \ds \notag \\
    &\le  c_1 \|\nabla v_0\|_{\leb\infty} 
          + c_1 \|g\|_{L^\infty((0, T); \leb2)} \int_0^T (1 + s^{-\frac12 - \frac22(\frac12 - \frac14)}) \ds \notag \\
    &\le  c_2
    \qquad \text{for all $t \in (0, T)$}
  \end{align}
  for some $c_1, c_2 > 0$.
  Similarly as in the proof of Lemma~\ref{lem;Linf-v}, we set
  \begin{align*}
    M_u(T') \defs \sup_{t \in (0, T')} \|u(\cdot, t)\|_{\leb\infty}
    \qquad \text{for $T' \in (0, T)$}.
  \end{align*}
  Then the maximum principle, another application of \cite[Lemma~3.1]{WinklerAggregationVsGlobal2010}, Hölder's inequality (we note that $\frac13 = \frac{1/6}{2} + \frac{5/6}{\infty} + \frac14$)
  and \eqref{eq:u_linfty:nabla_v_l4} imply
  \begin{align*}
    &\pe  \|u(\cdot, t)\|_{\leb \infty} \\
    &\le  \|\ure^{t \Du \Delta} u_0\|_{\leb\infty} 
          + \chiu \int_0^t \|\ure^{(t-s) \Du \Delta} \nabla \cdot (u \nabla v)(\cdot, s)\|_{\leb\infty} \ds
          + \int_0^t \betau \ds \\
    &\le  \|u_0\|_{\leb\infty} 
          + c_3 \chiu \|u\nabla v\|_{L^\infty((0, T); \leb3)} \int_0^T (1 + s^{-\frac12 - \frac{2}2(\frac13 - \frac1\infty)}) \ds
          + \betau T \\
    &\le  c_4 + c_4 \|u\|_{L^\infty((0, T); \leb2)}^\frac16 \|u\|_{L^\infty((0, T'); \leb\infty)}^\frac56 \|\nabla v\|_{L^\infty((0, T); \leb4)} \\
    &\le  c_4 + c_5 \big(M_u(T')\big)^\frac56
    \qquad \text{for all $t \in (0, T')$ and all $T' \in (0, T)$}
  \end{align*}
  for some $c_3, c_4, c_5 > 0$,
  Taking the supremum of all $t \in (0, T')$ shows $M_u(T') \le c_4 + c_5 \big(M_u(T')\big)^\frac56$ for all $T' \in (0, T)$
  and hence that there is $c_6 > 0$ with $M_u(T') \le c_6$ for all $T' \in (0, T)$,
  so that \eqref{eq:u_linfty:est} follows upon taking $T' \nea T$.
\end{proof}

Since the term $+\gammaw uv$ in the third equation in \eqref{P} is bounded by \eqref{eq:Linf-v:est} and \eqref{eq:u_linfty:est},
the subsystem solved by $(z, w)$ is structurally similar to that solved by $(u, v)$,
so that similar reasonings as above yield boundedness of $z$:
\begin{lem}\label{lm:z_linfty}
  Let $T \in (0, \infty) \cap (0, \tmax]$.
  Then there is $C > 0$ such that
  \begin{align}\label{eq:z_linfty:est}
    z \le C
    \qquad \text{in $(0, T)$}.
  \end{align}
\end{lem}
\begin{proof}
  Thanks to the $L^\infty$ bounds in \eqref{eq:Linf-v:est} and \eqref{eq:u_linfty:est},
  $\ol w(x, t) \defs \|w_0\|_{\leb\infty} + \gammaw \|uv\|_{L^\infty(\Omega \times (0, T))} \cdot t$, $(x, t) \in \Ombar \times [0, T)$, is a bounded supersolution of the third equation in \eqref{P},
  so that $w$ is bounded in $\Omega \times (0, T)$ by the comparison principle.
  A testing procedure as in Lemma~\ref{lm:u_l2} (utilising \eqref{eq:secondEn:spacetime}) then yields $z \in L^\infty((0, T); \leb2)$,
  whenceupon semigroup arguments as in Lemma~\ref{lm:u_linfty} provide $L^\infty$ estimates first for $\nabla w$ and then for $z$.
\end{proof}

Theorem~\ref{TH;2D} is now a straightforward consequence of these a~priori estimates.
\begin{proof}[Proof of Theorem~\ref{TH;2D}]
  If the solution $(u, v, w, z)$ of \eqref{P} given by Lemma~\ref{lm:local_ex} were not global in time,
  its first and fourth component would be bounded in $\Omega \times (0, \tmax)$ by \eqref{eq:u_linfty:est} and \eqref{eq:z_linfty:est}, respectively,
  which would contradict \eqref{eq:local_ex:ext_crit}.
\end{proof}

\section{Global weak solutions in 3D}\label{sec:3d}
Now turning our attention to the setting of a smooth, bounded domain $\Omega\subset ℝ^3$, we first modify the system so as to ensure global classical solvability and then, building on the estimates derived in Section~\ref{sec:combined}, construct solutions from compactness arguments.

Throughout this section, we fix $\zeta \in C^\infty(\R; [0, 1])$ with $\zeta(s) = 1$ for $s \le 1$ and $\zeta(s) = 0$ for $s \ge 2$ and then set
\begin{align*}
  \sigmae(s) = \int_0^s \zeta(\eps \tau) \dtau,
  \quad s \ge 0,
\end{align*}
for $\eps \in (0, 1)$,
whence
\begin{align}\label{eq:sigma_eps_prop}
  \sigmae(0) = 0,
  \quad
  0 \le \sigmae'
  \begin{cases}
    = 1 & \text{in $[0, \frac1\eps]$}, \\
    \in [0, 1] & \text{in $(\frac1\eps, \frac2\eps)$}, \\
    = 0 & \text{in $[\frac2\eps, \infty)$}
  \end{cases}
  \quad \text{and} \quad 
  \frac{s\sigmae'(s)}{\sigmae(s)}
  \begin{cases}
    = \sigmae'(s) = 1& \text{if $s \in [0, \frac1\eps]$}, \\
    \le \frac{\frac2\eps \cdot 1}{\frac1\eps} \le 2 & \text{if $s \in (\frac1\eps, \frac2\eps)$}, \\
    = 0 & \text{if $s \in [\frac2\eps, \infty)$},
  \end{cases}
\end{align}
so that $\sigmae \in \Sigmaset$ for all $\eps \in (0, 1)$, where $\Sigmaset$ is defined in \eqref{eq:Sigma}.

Moreover, we fix parameters as in \eqref{eq:intro:params}, $f$ as in \eqref{cond:f} and $u_0, v_0, w_0, z_0$ as in \eqref{eq:3d_main:init}
as well as families $(\une)_{\eps \in (0, 1)}$, $(\vne)_{\eps \in (0, 1)}$, $(\wne)_{\eps \in (0, 1)}$, $(\zne)_{\eps \in (0, 1)} \subset C^\infty(\Ombar)$ of positive functions
fulfilling \eqref{eq:def_A} for some ($\eps$-independent) $A > 0$ and being such that
\begin{align}\label{eq:conv_init}
  (\une, \vne, \wne, \zne) \to (u_0, v_0, w_0, z_0) &\qquad \text{in $\big(\leb1\big)^4$ as $\eps \sea 0$.}
\end{align}
(Due to \eqref{eq:3d_main:init}, the existence of such families can be seen by typical mollification arguments.)
Thanks to Lemma~\ref{lm:local_ex}, for each $\eps \in (0, 1)$, we may then further fix a smooth solution $(\ue, \ve, \we, \ze)$ of
\begin{align}\label{prob:eps}
  \begin{cases}
    \uet = \Du \Delta \ue - \chiu \nabla \cdot (\ue \sigmae'(\ue) \nabla \ve) - \gammau \sigma(\ue) \ve - \deltau \ue + \betau & \text{in $\Omega \times (0, \tmaxe)$}, \\
    \vet = \Dv \Delta \ve + \rhov \ve - \gammav \sigmae(\ue) \ve + \muv \we                                                     & \text{in $\Omega \times (0, \tmaxe)$}, \\
    \wet = \Dw \Delta \we + \gammaw \sigmae(\ue)\ve - \alphaw \we \sigmae(\ze) - \muw \we                                       & \text{in $\Omega \times (0, \tmaxe)$}, \\
    \zet = \Dz \Delta \ze - \chiz \nabla \cdot (\ze \sigmae'(\ze) \nabla \we) + \alphaz f(\we)\sigmae(\ze) - \deltaz \ze        & \text{in $\Omega \times (0, \tmaxe)$}, \\
    \partial_\nu \ue = \partial_\nu \ve = \partial_\nu \we = \partial_\nu \ze = 0                                               & \text{on $\partial \Omega \times (0, \tmaxe)$}, \\
    (\ue, \ve, \we, \ze)(\cdot, 0) = (\une, \vne, \wne, \zne)                                                                   & \text{in $\Omega$}, \\
  \end{cases}
\end{align}
with maximal existence time $\tmaxe \in (0, \infty]$.
We emphasize that since $\sigmae \in \Sigmaset$ and as all initial data fulfil \eqref{eq:def_A} with a uniform constant $A$,
the results from Section~\ref{sec:combined} imply $\eps$-independent estimates for the solutions of \eqref{prob:eps}.

Due to the regularization introduced by the compactly supported functions $\sigmae'$, these solutions are rapidly seen to be global in time.
\begin{lem}\label{lm:tmax_eps_inf}
  For all $\eps \in (0, 1)$, we have $\tmaxe = \infty$. 
\end{lem}
\begin{proof}
  Suppose $\tmaxe < \infty$ for some $\eps \in (0, 1)$.
  As $\sigmae' = 0$ in $[\frac2\eps, \infty)$,
  since $\sigmae \le \frac2\eps$ in $[0, \infty)$ by \eqref{eq:sigma_eps_prop}
  and as $\ve$ is bounded by Lemma~\ref{lem;Linf-v},
  the functions $\ol{\ue}$, $\ol{\we}$ and $\ol{\ze}$ defined by
  \begin{align*}
    \ol{\ue}(x, t) &\defs \max\left\{\|\une\|_{\leb\infty}, \frac2\eps\right\} + \betau \cdot t, \\[0.3em]
    \ol{\we}(x, t) &\defs \|\wne\|_{\leb\infty} \exp \left( \frac{2\gammaw \|\ve\|_{L^\infty(\Omega \times (0, \tmaxe))}}{\eps} \cdot t \right), \\[0.3em]
    \ol{\ze}(x, t) &\defs \max\left\{\|\zne\|_{\leb\infty}, \frac2\eps\right\} + \frac{2\alphaz \|\ol{\we}\|_{L^\infty(\Omega \times (0, \tmaxe))}}{\eps} \cdot t
  \end{align*}
  for $(x, t) \in \Ombar \times [0, \tmaxe)$ are bounded supersolutions of the first, third and fourth equation in \eqref{prob:eps}, respectively.
  By the comparison principle (and since nonnegativity is already proven in Lemma~\ref{lm:local_ex}),
  $\ue$ and $\ze$ would then also be bounded in $\Omega \times (0, \tmaxe)$, contradicting \eqref{eq:local_ex:ext_crit}.
\end{proof}

Next, we note that the a~priori estimates from Section~\ref{sec:combined} and Lemma~\ref{lm:space_time_interpol} imply the following space-time integrability properties.
\begin{lem}\label{lm:spacetime_est}
  Let $T > 0$ and $\eta \in (0, \frac54)$.
  Then there is $C > 0$ such that 
  {\allowdisplaybreaks\begin{align}\label{eq:spacetime_est:u}
     &\intntom \ue^\frac53 + \intntom |\nabla \ue|^\frac54 \le C, \\
     &\|\ve\|_{L^\infty(\Omega \times (0, T))} + \intntom |\nabla \ve|^4 \le C, \label{eq:spacetime_est:v}\\
     &\intntom \we^{5-\eta} + \intntom |\nabla \we|^{\frac52-\eta} \le C,  \label{eq:spacetime_est:w}\\
     &\intntom \ze^\frac53 + \intntom |\nabla \ze|^\frac54 \le C \quad \text{and} \label{eq:spacetime_est:z} \\
     &\intntom (\ue \ve)^\frac53 + \intntom (\we \ze)^{\frac54-\eta} \le C \label{eq:spacetime_est:uv_wz}
  \end{align}}%
  for all $\eps \in (0, 1)$.
\end{lem}
\begin{proof}
  According to Lemma~\ref{lem;L1Estimate} and Lemma~\ref{lem;firstEn},
  the families $(\ue)_{\eps \in (0, 1)}$ and $(\sqrt{\ue})_{\eps \in (0, 1)}$ are bounded in $L^\infty((0, T); \leb1)$ and $L^2((0, T); \sob12)$, respectively,
  so that \eqref{eq:space_time_interpol:special} yields \eqref{eq:spacetime_est:u}.
  Analogously, but relying on Lemma~\ref{lem;secondEn} instead of Lemma~\ref{lem;firstEn}, we obtain \eqref{eq:spacetime_est:z}.
  The estimate for the first term in \eqref{eq:spacetime_est:v} is contained in \eqref{eq:Linf-v:est}, whenceupon the second one then follows from \eqref{eq:firstEn:spacetime} and \eqref{eq:nabla_varphi_l14:est1}.
    
  Since $\sup_{\eps \in (0, 1)} \sup_{t \in (0, T)} \intom \we^{3-\eta}(\cdot, t) < \infty$ for all $\eta \in (0, 3)$ by \eqref{eq:Linf-v:est}
  and $\sup_{\eps \in (0, 1)} \intntom |\nabla \we^\frac14|^4 < \infty$ by \eqref{eq:secondEn:spacetime} and \eqref{eq:nabla_varphi_l14:est1},  also the estimate in \eqref{eq:spacetime_est:w} follows from Lemma~\ref{lm:space_time_interpol}, because $\displaystyle\lim_{\eta \sea 0} \frac{4(\frac34+3-\eta)}{3} = 5$ and $\displaystyle\lim_{\eta \sea 0} \frac{4(\frac34+3-\eta)}{3+3-\eta} = \frac52$. 
  Finally, \eqref{eq:spacetime_est:uv_wz} results from \eqref{eq:spacetime_est:u}--\eqref{eq:spacetime_est:z} and Hölder's inequality.
\end{proof}

As a final preparation for applications of the Aubin--Lions lemma, we obtain suitable a~priori estimates for the time derivatives.
\begin{lem}\label{lm:time_deriv}
  Let $T \in (0, \infty)$ and set $X \defs L^1((0, T); (\sob32)^\star)$.
  Then there exists $C > 0$ such that
  \begin{align*}
    \|\uet\|_X + \|\vet\|_X + \|\wet\|_X +  \|\zet\|_X \le C
    \qquad \text{for all $\eps \in (0, 1)$}.
  \end{align*}
\end{lem}
\begin{proof}
  Since for each term on the right-hand sides of the first three equations in \eqref{prob:eps}
  our a~priori estimates are at least as strong as for a term on the right-hand side of the fourth equation,
  it suffices to discuss the bound for $\zet$.
  By testing the fourth equation in \eqref{prob:eps} with $\varphi \in\con\infty$, integrating by parts and applying Hölder's inequality, we see that
  \begin{align*}
          \left|\intom \zet \varphi \right|
    &\le  \Dz \left|\intom \nabla \ze \cdot \nabla \varphi \right|
          + \chiz \left|\intom \ze \sigma_\eps'(\ze) \nabla \we \cdot \nabla \varphi \right|
          + \alphaz \left|\intom f(\we) \sigma_\eps(\ze) \varphi \right| 
          + \deltaz \left|\intom \ze \varphi \right|\\
    &\le  \left(
            \Dz \|\nabla \ze\|_{\leb1}
            + \chiz \|\ze \sigma_\eps'(\ze) \nabla \we\|_{\leb1}
            + \alphaz \|\we \ze\|_{\leb1}
            + \deltaz \|\ze\|_{\leb1}
          \right) \|\varphi\|_{\sob1\infty}
  \end{align*}
  holds for all $\varphi \in \con\infty$ and all $\eps \in (0, 1)$.
  The estimates in \eqref{eq:sigma_eps_prop}, \eqref{eq:secondEn:spacetime} and \eqref{eq:spacetime_est:uv_wz} entail a bound for
  $|\ze σ_{ε}'(\ze) ∇\we| = \f{\ze σ_{ε}'(\ze)}{σ_{ε}(\ze)} σ_\eps(\ze)|∇\we| \le 2 \sqrt{\we \ze}\cdot \f{\sqrt{σ_{ε}(\ze)}}{\sqrt{\we}}|∇\we|$ in $L^1(\Omega \times (0, T))$. Therefore  and due to \eqref{eq:spacetime_est:z}, \eqref{eq:spacetime_est:uv_wz} and the embedding $\sob32 \embed \sob1\infty$,
  there is $c_1 > 0$ such that
  \begin{align*}
    \intnt \left| \intom \zet \varphi \right| \le c_1 \|\varphi\|_{\sob32}
    \qquad \text{for all $\varphi \in \con\infty$ and all $\eps \in (0, 1)$},
  \end{align*}
  that is, $\|\zet\|_{X} \le c_1$ for all $\eps \in (0, 1)$.
\end{proof}

\begin{lem}\label{lm:eps_sea_0}
  There exist a null sequence $(\eps_j)_{j \in \N} \subset (0, 1)$ and functions $u, v, w, z \colon \Omega \times (0, \infty) \to [0, \infty)$ such that
  {\allowdisplaybreaks
  \begin{alignat}{2}
    (\ue, \ve, \we, \ze) &\to (u, v, w, z) &&\qquad \text{in $\big(L_{\loc}^1(\Ombarinf)\big)^4$ and a.e.\ in $\Omega \times (0, \infty)$}, \label{eq:eps_sea_0:l1_pw}\\
    (\ue, \ve, \we, \ze) &\rh (u, v, w, z) &&\qquad \text{in $\big(L_{\loc}^\frac54([0, \infty); \sob1{\frac54})\big)^4$}, \label{eq:eps_sea_0:w154} \\
    \sigmae(\ue) \ve &\to uv &&\qquad \text{in $L_{\loc}^1(\Ombarinf)$}, \label{eq:eps_sea_0:uv}\\
    f(\we) \sigmae(\ze) &\to f(w)z &&\qquad \text{in $L_{\loc}^1(\Ombarinf)$}, \label{eq:eps_sea_0:fwz}\\
    \we \sigmae(\ze) &\to wz &&\qquad \text{in $L_{\loc}^1(\Ombarinf)$}, \label{eq:eps_sea_0:wz}\\
    \frac{\sqrt{\sigmae(\ue)}}{\sqrt{\ve}}\nabla \ve &\rh \frac{\sqrt{u}}{\sqrt{v}}\nabla v &&\qquad \text{in $L_{\loc}^2(\Ombarinf)$}, \label{eq:eps_sea_0:uv_nabla_v} \\
    \frac{\sqrt{\sigmae(\ze)}}{\sqrt{\we}}\nabla \we &\rh \frac{\sqrt{z}}{\sqrt{w}}\nabla w &&\qquad \text{in $L_{\loc}^2(\Ombarinf)$}, \label{eq:eps_sea_0:wz_nabla_w}
  \end{alignat}}%
  as $\eps = \eps_j \sea 0$.
\end{lem}
\begin{proof}
  Since $\sob1{\frac54}$ embeds compactly into $\leb1$ which embeds into $(\sob32)^\star$,
  and due to Lemma~\ref{lm:spacetime_est} and Lemma~\ref{lm:time_deriv},
  the Aubin--Lions lemma and a diagonalization argument yield a null sequence $(\eps_j)_{j \in \N}$ such that $(\ue, \ve, \we, \ze) \to (u, v, w, z)$ in $\leb1$ as $\eps = \eps_j \sea 0$.
  By switching to a subsequence, this implies \eqref{eq:eps_sea_0:l1_pw},
  whenceupon another application of Lemma~\ref{lm:spacetime_est} gives \eqref{eq:eps_sea_0:w154}.
  Moreover, since $\ue$, $\ve$, $\we$, $\ze$ are nonnegative by Lemma~\ref{lm:local_ex}, this implies nonnegativity of the limit functions.
  Since \eqref{eq:sigma_eps_prop} and \eqref{eq:spacetime_est:uv_wz} entail equi-integrability of $(\sigmae(\ue) \ve)_{\eps \in (0, 1)}$ and $(\we \sigmae(\ze))_{\eps \in (0, 1)}$ in $\Omega \times (0, T)$ for all $T > 0$
  by the de la Vall\'ee Poussin theorem,
  Vitali's theorem and the pointwise a.e.\ convergence in \eqref{eq:eps_sea_0:l1_pw} imply \eqref{eq:eps_sea_0:uv}, \eqref{eq:eps_sea_0:fwz} and \eqref{eq:eps_sea_0:wz}.

  By \eqref{eq:secondEn:spacetime}, $\Big(2\sqrt{\sigma_{\eps_j}(z_{\eps_j})} \cdot \nabla \sqrt{w_{\eps_j}\vphantom{(z_{\eps_j})}}\,\Big)_{j \in \N}$ has a subsequence that weakly converges in $L_{\loc}^2(\Ombarinf)$. Its limit must coincide with $2\sqrt{z}∇\sqrt{w}=\f{\sqrt{z}}{\sqrt{w}}∇w$: 
  Since $\sqrt{\we} \to \sqrt w$ a.e.\ by \eqref{eq:eps_sea_0:l1_pw}, Lemma~\ref{lem;secondEn} asserts $\nabla \sqrt{\we} \rh \nabla \sqrt{w}$ in $L_{\loc}^2(\Ombarinf)$ as $\eps = \eps_j \sea 0$, and together with \eqref{eq:eps_sea_0:l1_pw}, this shows $\sqrt{\sigmae(\ze)} \cdot \nabla \sqrt{\we\vphantom{(\ze)}} \rh \sqrt{z} \cdot \nabla \sqrt{w}$ in $L_{\loc}^1(\Ombarinf)$ as $\eps = \eps_j \sea 0$. 
 Finally, \eqref{eq:eps_sea_0:uv_nabla_v} can be shown analogously.
\end{proof}

Since $(\ue, \ve, \we, \ze)$ solves \eqref{prob:eps},
the convergence statements in Lemma~\ref{lm:eps_sea_0} allow us to conclude that the limit function obtained there is a weak solution of \eqref{P} in the following sense.
\begin{df}\label{def:weak_sol}
  Let $u_0, v_0, w_0, z_0 \in \leb1$ be nonnegative.
  A quadruple $(u, v, w, z) \in \big(L_{\loc}^1(\Ombarinf; [0, \infty))\big)^4$ with $uv, wz, u \nabla v, z \nabla w \in L_{\loc}^1(\Ombarinf)$ is called a \emph{weak solution} of \eqref{P} if
  {\allowdisplaybreaks\begin{align*}
          - \intninfom u \varphi_t
          - \intom u_0 \varphi(\cdot, 0)
    &=    - \Du \intninfom \nabla u \cdot \nabla \varphi
          + \chiu \intninfom u \nabla v \cdot \nabla \varphi \\
    &\pe  -\, \gammau \intninfom u v \varphi
          - \deltau \intninfom u \varphi
          + \betau \intninfom \varphi, \\
          - \intninfom v \varphi_t
          - \intom v_0 \varphi(\cdot, 0)
    &=    - Dv \intninfom \nabla v \cdot \nabla \varphi \\
    &\pe  +\, \rhov \intninfom v \varphi
          - \gammav \intninfom uv \varphi
          + \muv \intninfom w \varphi, \\
          - \intninfom w \varphi_t
          - \intom w_0 \varphi(\cdot, 0)
    &=    - \Dw \intninfom \nabla w \cdot \nabla \varphi \\
    &\pe  +\, \gammaw \intninfom uv \varphi
          - \alphaw \intninfom w z \varphi
          - \mu_w \intninfom w \varphi, \\
          - \intninfom z \varphi_t
          - \intom z_0 \varphi(\cdot, 0)
    &=    - \Dz \intninfom \nabla z \cdot \nabla \varphi
          + \chiz \intninfom z \nabla w \cdot \nabla \varphi \\
    &\pe  +\, \alphaz \intninfom f(w) z \varphi
          - \deltaz \intninfom z \varphi
  \end{align*}}
  hold for all $\varphi \in C_c^\infty(\Ombarinf)$.
\end{df}

\begin{lem}\label{lm:uvwz_weak_sol}
  The quadruple $(u, v, w, z)$ constructed in Lemma~\ref{lm:eps_sea_0} is a weak solution of \eqref{P} in the sense of Definition~\ref{def:weak_sol}.
\end{lem}
\begin{proof}
  The required regularity information is implicitly contained in Lemma~\ref{lm:eps_sea_0}.
  Let $\varphi \in C_c^\infty(\Ombar)$ and let $(\eps_j)_{j \in \N}$ be as given by Lemma~\ref{lm:eps_sea_0}.
  As $\eps = \eps_j \to 0$, we then have
  \begin{align*}
    \intninfom \ze \varphi_t \to \intninfom z \varphi_t
    \quad \text{and} \quad
    \intninfom \ze \varphi \to \intninfom z \varphi
  \end{align*}
  by \eqref{eq:eps_sea_0:l1_pw},
  \begin{align*}
   \Dz \intninfom \nabla \ze \cdot \nabla \varphi \to \Dz\intninfom \nabla z \cdot \nabla \varphi
  \end{align*}
  by \eqref{eq:eps_sea_0:w154}
  \begin{align*}
   \alphaz \intninfom f(\we) \sigmae(\ze) \varphi \to \alphaz\intninfom f(w) z \varphi
  \end{align*}
  by \eqref{eq:eps_sea_0:fwz} and
  \begin{align*}
   \deltaz \intom \zne \varphi(\cdot, 0) \to \deltaz \intom z_0 \varphi(\cdot, 0)
  \end{align*}
  by \eqref{eq:conv_init}. 
  As to convergence of the taxis term,
  we first note that $(\frac{\ze \sigmae'(\ze)}{\sigmae(\ze)})_{\eps \in (0, 1)}$ is bounded by $2$ and converges a.e.\ to $1$ as $\eps = \eps_j \sea 0$ by \eqref{eq:Sigma} and \eqref{eq:eps_sea_0:l1_pw},
  so that $\sqrt{\we \sigmae(\ze)} \frac{\ze \sigmae'(\ze)}{\sigmae(\ze)} \to \sqrt{zw}$ in $L^2(\Ombarinf)$ as $\eps = \eps_j \sea 0$ thanks to \eqref{eq:eps_sea_0:wz} and Lebesgue's theorem.
  Together with \eqref{eq:eps_sea_0:wz_nabla_w}, this implies
  \begin{align*}
        \intninfom \ze \sigmae'(\ze) \nabla \we \cdot \nabla \varphi
    =   \intninfom \sqrt{\we \sigmae(\ze)} \frac{\ze \sigmae'(\ze)}{\sigmae(\ze)} \cdot \frac{\sqrt{\sigmae(\ze)}}{\sqrt{\we}} \nabla \we \cdot \nabla \varphi
    \to \intninfom z \nabla w \cdot \nabla \varphi
  \end{align*}
  as $\eps = \eps_j \sea 0$.
  In conclusion, $z$ solves the corresponding subproblem of \eqref{P} weakly.
  The terms for the remaining solution components can be treated in a similar fashion.
\end{proof}

\begin{proof}[Proof of Theorem~\ref{TH;3D}]
  All claims have been proven in Lemma~\ref{lm:uvwz_weak_sol}.
\end{proof}

\addcontentsline{toc}{section}{References}

\footnotesize
\begin{thebibliography}{10}
\setlength{\itemsep}{0.2pt}

\bibitem{AmannNonhomogeneousLinearQuasilinear1993}
\textsc{Amann, H.}:
\newblock {\em Nonhomogeneous linear and quasilinear elliptic and
  pa\-ra\-bo\-lic boundary value problems}.
\newblock In \textsc{Schmeisser, H.} and \textsc{Triebel, H.}, editors, {\em
  Function {{Spaces}}, {{Differential Operators}} and {{Nonlinear Analysis}}},
  \href{https://doi.org/10.1007/978-3-663-11336-2_1}{pages 9--126}.
  Vieweg+Teubner Verlag, Wiesbaden, 1993.
\newblock

\bibitem{antia1996models}
\textsc{Antia, R.}, \textsc{Koella, J.~C.}, and \textsc{Perrot, V.}:
\newblock {\em Models of the within-host dynamics of persistent mycobacterial
  infections}.
\newblock Proceedings of the Royal Society of London. Series B: Biological
  Sciences, 263(1368):257--263, 1996.

\bibitem{BellomoEtAlMathematicalTheoryKeller2015}
\textsc{Bellomo, N.}, \textsc{Bellouquid, A.}, \textsc{Tao, Y.}, and
  \textsc{Winkler, M.}:
\newblock {\em Toward a mathematical theory of {{Keller}}--{{Segel}} models of
  pattern formation in biological tissues}.
\newblock Math. Models Methods Appl. Sci.,
  \href{https://doi.org/10.1142/S021820251550044X}{25(09):1663--1763}, 2015.
\newblock

\bibitem{castillo-chavez_song}
\textsc{Castillo-Chavez, C.} and \textsc{Song, B.}:
\newblock {\em Dynamical models of tuberculosis and their applications}.
\newblock Math. Biosci. Eng.,
  \href{https://doi.org/10.3934/mbe.2004.1.361}{1(2):361--404}, 2004.
\newblock

\bibitem{catalaPLOS}
\textsc{Catal{\`a}, M.}, \textsc{Prats, C.}, \textsc{L{\'o}pez, D.},
  \textsc{Cardona, P.-J.}, and \textsc{Alonso, S.}:
\newblock {\em A reaction-diffusion model to understand granulomas formation
  inside secondary lobule during tuberculosis infection}.
\newblock PloS one,
  \href{https://doi.org/10.1371/journal.pone.0239289}{15(9):e0239289}, 2020.
\newblock

\bibitem{ancient_tuberculosis}
\textsc{Cave, A.} and \textsc{Demonstrator, A.}:
\newblock {\em The evidence for the incidence of tuberculosis in ancient
  {E}gypt}.
\newblock British Journal of Tuberculosis, 33(3):142--152, 1939.

\bibitem{chakraborty2024brief}
\textsc{Chakraborty, D.}, \textsc{Batabyal, S.}, and \textsc{Ganusov, V.~V.}:
\newblock {\em A brief overview of mathematical modeling of the within-host
  dynamics of mycobacterium tuberculosis}.
\newblock Frontiers in Applied Mathematics and Statistics, 10:1355373, 2024.

\bibitem{CieslakEtAlExistenceGlobalSolutions2024}
\textsc{Cie{\'s}lak, T.}, \textsc{Fuest, M.}, \textsc{Hajduk, K.}, and
  \textsc{Sier{\.z}{\k e}ga, M.}:
\newblock {\em On the existence of global solutions for the {{3D}}
  chemorepulsion system}.
\newblock Z. Anal. Anwend.,
  \href{https://doi.org/10.4171/zaa/1747}{43(1):49--65}, 2024.
\newblock

\bibitem{DuanEtAlGlobalSolutionsCoupled2010}
\textsc{Duan, R.}, \textsc{Lorz, A.}, and \textsc{Markowich, P.}:
\newblock {\em Global solutions to the coupled chemotaxis-fluid equations}.
\newblock Commun. Partial Differ. Equ.,
  \href{https://doi.org/10.1080/03605302.2010.497199}{35(9):1635--1673}, 2010.
\newblock

\bibitem{ehlers_schaible}
\textsc{Ehlers, S.} and \textsc{Schaible, U.~E.}:
\newblock {\em The granuloma in tuberculosis: dynamics of a host--pathogen
  collusion}.
\newblock Frontiers in immunology, 3:411, 2013.

\bibitem{feng}
\textsc{Feng, P.}:
\newblock {\em A spatial model to understand tuberculosis granuloma formation
  and its impact on disease progression}.
\newblock Journal of Nonlinear, Complex and Data Science, 25(1):19--35, 2024.

\bibitem{feng2000model}
\textsc{Feng, Z.}, \textsc{Castillo-Chavez, C.}, and \textsc{Capurro, A.~F.}:
\newblock {\em A model for tuberculosis with exogenous reinfection}.
\newblock Theoretical population biology, 57(3):235--247, 2000.

\bibitem{feng_ianelli_milner}
\textsc{Feng, Z.}, \textsc{Iannelli, M.}, and \textsc{Milner, F.~A.}:
\newblock {\em A two-strain tuberculosis model with age of infection}.
\newblock SIAM J. Appl. Math.,
  \href{https://doi.org/10.1137/S003613990038205X}{62(5):1634--1656}, 2002.
\newblock

\bibitem{friedman_lam}
\textsc{Friedman, A.} and \textsc{Lam, K.-Y.}:
\newblock {\em On the stability of steady states in a granuloma model}.
\newblock J. Differential Equations,
  \href{https://doi.org/10.1016/j.jde.2014.02.019}{256(11):3743--3769}, 2014.
\newblock

\bibitem{gammack_doering_kirschner}
\textsc{Gammack, D.}, \textsc{Doering, C.~R.}, and \textsc{Kirschner, D.~E.}:
\newblock {\em Macrophage response to \textit{ {M}ycobacterium tuberculosis}
  infection}.
\newblock J. Math. Biol.,
  \href{https://doi.org/10.1007/s00285-003-0232-8}{48(2):218--242}, 2004.
\newblock

\bibitem{gammack_ganguli_marino_kirschner}
\textsc{Gammack, D.}, \textsc{Ganguli, S.}, \textsc{Marino, S.},
  \textsc{Segovia-Juarez, J.}, and \textsc{Kirschner, D.~E.}:
\newblock {\em Understanding the immune response in tuberculosis using
  different mathematical models and biological scales}.
\newblock Multiscale Model. Simul.,
  \href{https://doi.org/10.1137/040603127}{3(2):312--345}, 2005.
\newblock

\bibitem{hao_schlesinger_friedman}
\textsc{Hao, W.}, \textsc{Schlesinger, L.~S.}, and \textsc{Friedman, A.}:
\newblock {\em Modeling granulomas in response to infection in the lung}.
\newblock PLoS One, 11(3):e0148738, 2016.

\bibitem{hoerter}
\textsc{Hoerter, A.}, \textsc{Arnett, E.}, \textsc{Schlesinger, L.~S.}, and
  \textsc{Pienaar, E.}:
\newblock {\em Systems biology approaches to investigate the role of granulomas
  in {TB-HIV} coinfection}.
\newblock Frontiers in Immunology, 13:1014515, 2022.

\bibitem{houben_dodd}
\textsc{Houben, R.~M.} and \textsc{Dodd, P.~J.}:
\newblock {\em The global burden of latent tuberculosis infection: a
  re-estimation using mathematical modelling}.
\newblock PLoS medicine,
  \href{https://doi.org/10.1371/journal.pmed.1002152}{13(10):e1002152}, 2016.
\newblock

\bibitem{ibarguen_esteva_burbano}
\textsc{Ibarg\"uen-Mondrag\'on, E.}, \textsc{Esteva, L.}, and
  \textsc{Burbano-Rosero, E.~M.}:
\newblock {\em Mathematical model for the growth of {M}ycobacterium
  tuberculosis in the granuloma}.
\newblock Math. Biosci. Eng.,
  \href{https://doi.org/10.3934/mbe.2018018}{15(2):407--428}, 2018.
\newblock

\bibitem{ibarguen_esteva_chavez}
\textsc{Ibarguen-Mondragon, E.}, \textsc{Esteva, L.}, and
  \textsc{Ch\'avez-Gal\'an, L.}:
\newblock {\em A mathematical model for cellular immunology of tuberculosis}.
\newblock Math. Biosci. Eng.,
  \href{https://doi.org/10.3934/mbe.2011.8.973}{8(4):973--986}, 2011.
\newblock

\bibitem{JiangEtAlGlobalExistenceAsymptotic2015}
\textsc{Jiang, J.}, \textsc{Wu, H.}, and \textsc{Zheng, S.}:
\newblock {\em Global existence and asymptotic behavior of solutions to a
  chemotaxis-fluid system on general bounded domains}.
\newblock Asymptot. Anal.,
  \href{https://doi.org/10.3233/asy-141276}{92(3-4):249--258}, 2015.
\newblock

\bibitem{KellerSegelTravelingBandsChemotactic1971}
\textsc{Keller, E.~F.} and \textsc{Segel, L.~A.}:
\newblock {\em Traveling bands of chemotactic bacteria: A theoretical
  analysis}.
\newblock J. Theor. Biol.,
  \href{https://doi.org/10.1016/0022-5193(71)90051-8}{30(2):235--248}, 1971.
\newblock

\bibitem{kirschner1999dynamics}
\textsc{Kirschner, D.}:
\newblock {\em Dynamics of co-infection with {{M.}} tuberculosisand {{HIV-1}}}.
\newblock Theoretical population biology, 55(1):94--109, 1999.

\bibitem{LadyzenskajaSolutionLargeNonstationary1959}
\textsc{Lady{\v z}enskaja, O.~A.}:
\newblock {\em Solution ``in the large'' of the nonstationary boundary value
  problem for the {{Navier-Stokes}} system with two space variables}.
\newblock Commun. Pure Appl. Math.,
  \href{https://doi.org/10.1002/cpa.3160120303}{12:427--433}, 1959.
\newblock

\bibitem{LankeitWinklerDepletingSignalAnalysis2023}
\textsc{Lankeit, J.} and \textsc{Winkler, M.}:
\newblock {\em Depleting the signal: {{Analysis}} of chemotaxis-consumption
  models -- {{A}} survey}.
\newblock Stud. Appl. Math.,
  \href{https://doi.org/10.1111/sapm.12625}{151(4):1197--1229}, 2023.
\newblock

\bibitem{LiLankeitBoundednessChemotaxisHaptotaxis2016}
\textsc{Li, Y.} and \textsc{Lankeit, J.}:
\newblock {\em Boundedness in a chemotaxis--haptotaxis model with nonlinear
  diffusion}.
\newblock Nonlinearity,
  \href{https://doi.org/10.1088/0951-7715/29/5/1564}{29(5):1564--1595}, 2016.
\newblock

\bibitem{magombedze}
\textsc{Magombedze, G.}, \textsc{Garira, W.}, and \textsc{Mwenje, E.}:
\newblock {\em Modelling the human immune response mechanisms to \textit{
  {M}ycobacterium tuberculosis} infection in the lungs}.
\newblock Math. Biosci. Eng.,
  \href{https://doi.org/10.3934/mbe.2006.3.661}{3(4):661--682}, 2006.
\newblock

\bibitem{ozacglar}
\textsc{Ozcaglar, C.}, \textsc{Shabbeer, A.}, \textsc{Vandenberg, S.~L.},
  \textsc{Yener, B.}, and \textsc{Bennett, K.~P.}:
\newblock {\em Epidemiological models of {\textit{ {m}ycobacterium
  tuberculosis}} complex infections}.
\newblock Math. Biosci.,
  \href{https://doi.org/10.1016/j.mbs.2012.02.003}{236(2):77--96}, 2012.
\newblock

\bibitem{QuittnerSoupletSuperlinearParabolicProblems2019}
\textsc{Quittner, P.} and \textsc{Souplet, P.}:
\newblock {\em Superlinear parabolic problems}.
\newblock Birkh\"auser Advanced Texts. Birkh\"auser/Springer, Cham, second
  edition, 2019.
\newblock Blow-up, global existence and steady states.
\newblock

\bibitem{ramakrishnan2012revisiting}
\textsc{Ramakrishnan, L.}:
\newblock {\em Revisiting the role of the granuloma in tuberculosis}.
\newblock Nature Reviews Immunology, 12(5):352--366, 2012.

\bibitem{su_zhou_dorman_jones}
\textsc{Su, B.}, \textsc{Zhou, W.}, \textsc{Dorman, K.~S.}, and \textsc{Jones,
  D.~E.}:
\newblock {\em Mathematical modelling of immune response in tissues}.
\newblock Comput. Math. Methods Med.,
  \href{https://doi.org/10.1080/17486700801982713}{10(1):9--38}, 2009.
\newblock

\bibitem{waaler}
\textsc{Waaler, H.}, \textsc{Geser, A.}, and \textsc{Andersen, S.}:
\newblock {\em The use of mathematical models in the study of the epidemiology
  of tuberculosis}.
\newblock American Journal of Public Health and the Nations Health,
  52(6):1002--1013, 1962.

\bibitem{wigginton_kirschner}
\textsc{Wigginton, J.~E.} and \textsc{Kirschner, D.}:
\newblock {\em A model to predict cell-mediated immune regulatory mechanisms
  during human infection with mycobacterium tuberculosis}.
\newblock The Journal of Immunology, 166(3):1951--1967, 2001.

\bibitem{WinklerAggregationVsGlobal2010}
\textsc{Winkler, M.}:
\newblock {\em Aggregation vs. global diffusive behavior in the
  higher-dimensional {{Keller}}--{{Segel}} model}.
\newblock J. Differ. Equ.,
  \href{https://doi.org/10.1016/j.jde.2010.02.008}{248(12):2889--2905}, 2010.
\newblock

\bibitem{WinklerGlobalLargedataSolutions2012}
\textsc{Winkler, M.}:
\newblock {\em Global large-data solutions in a
  chemotaxis--({{Navier--}}){{Stokes}} system modeling cellular swimming in
  fluid drops}.
\newblock Commun. Partial Differ. Equ.,
  \href{https://doi.org/10.1080/03605302.2011.591865}{37(2):319--351}, 2012.
\newblock

\bibitem{WinklerLogarithmicallyRefinedGagliardoNirenberg2023}
\textsc{Winkler, M.}:
\newblock {\em Logarithmically refined {{Gagliardo-Nirenberg}} interpolation
  and application to blow-up exclusion in a two-dimensional
  chemotaxis-consumption system}.
\newblock Preprint, 2023.

\bibitem{who_tuberculosis_rept2023}
\textsc{{World~Health~Organization}}:
\newblock {\em Global tuberculosis report 2023}.
\newblock Geneva, 2023.

\bibitem{zhang_liu_feng_jin}
\textsc{Zhang, R.}, \textsc{Liu, L.}, \textsc{Feng, X.}, and \textsc{Jin, Z.}:
\newblock {\em Existence of traveling wave solutions for a diffusive
  tuberculosis model with fast and slow progression}.
\newblock Appl. Math. Lett.,
  \href{https://doi.org/10.1016/j.aml.2020.106848}{112:7}, 2021.
\newblock Id/No 106848.
\newblock

\end{thebibliography}

\end{document}